\documentclass[11pt]{amsart}

\usepackage[active]{srcltx}
\usepackage{amsmath,amssymb}
\usepackage{amsthm}
\usepackage{hyperref}
\usepackage{latexsym}
\usepackage{color}
\usepackage{dsfont}
\usepackage[normalem]{ulem}
\usepackage{epsfig}
\usepackage{pgf}
\newtheorem{teo}{Theorem}[section]
\newtheorem{lema}[teo]{Lemma}
\newtheorem{prop}[teo]{Proposition}
\newtheorem{defi}{Definition}[section]
\newtheorem{cor}[teo]{Corollary}
\newtheorem{rem}{Remark}[section]

\renewcommand{\theequation}{\arabic{section}.\arabic{equation}}

\begin{document}

\title{Blow-up for a double nonlocal heat equation}

\author{R. Ferreira and A. de Pablo}

\address{Ra\'{u}l Ferreira
\hfill\break\indent  Departamento de An\'alisis Matem\'atico y Matem\'{a}tica Aplicada,
\hfill\break\indent Universidad Complutense de Madrid,
\hfill\break\indent
 28040 Madrid, Spain.
\hfill\break\indent  e-mail: {\tt raul$_-$ferreira@mat.ucm.es}}

\address{Arturo de Pablo
\hfill\break\indent  Departamento de Matem\'{a}ticas,
\hfill\break\indent Universidad Carlos III de Madrid,
\hfill\break\indent
 28911 Legan\'{e}s, Spain.
\hfill\break\indent  and ICMAT, Instituto de Ciencias Matem\'aticas (CSIC-UAM-UC3M-UCM),
\hfill\break\indent
 28049 Madrid , Spain.
 \hfill\break\indent  e-mail: {\tt arturop@math.uc3m.es}}

\

\begin{abstract}
We study the blow-up question for the diffusion equation
$$
(\mathcal{D}_t +\mathcal{L}_x) u=u^p,\qquad x\in\mathbb{R}^N,\;0<t<T,
$$
with $p>1$, where $\mathcal{D}_t$ is a nonlocal derivative in time defined by convolution with a nonnegative and nonincreasing kernel, and $\mathcal{L}_x$ is a nonlocal operator in space driven by a nonnegative radial L\'evy kernel.  We show that the existence of solutions that blow up in finite time or exist globally depends only on the behaviour of  the spatial kernel at infinity.  A main difficulty of the work stems from estimating the fundamental pair defining the solution through a Duhamel formula, due to the generality of the setting, which includes singular or not, at the origin, spatial kernels, that can be either positive or compactly supported.

As a byproduct we obtain that the Fujita exponent for the fractional type operators $\mathcal{D}_t\sim\partial_t^\alpha$, $0<\alpha<1$, and $\mathcal{L}_x\sim(-\Delta)^{\beta/2}$, $0<\beta<2$ (the Caputo fractional derivative and the fractional Laplacian, resp.), is $p_*=1+\beta/N$, provided $N<2\beta$ when $\mathcal{D}_t\neq\partial_t^\alpha$.

\end{abstract}

\maketitle



\section{Introduction}

\label{sect-introduction} \setcounter{equation}{0}

We consider non-negative bounded solutions to the following double nonlocal problem
\begin{equation}\label{eq.principal}
\left\{
\begin{array}{ll}
(\mathcal{D}_t +\mathcal{L}_x) u=u^p,\qquad & x\in \mathbb R^N,\; t\in(0,T),\\
u(x,0)=u_0(x),
\end{array}\right.
\end{equation}
where $N\ge1$, $p>1$, and $u_0\in L^1(\mathbb{R}^N)\cap L^\infty(\mathbb{R}^N)$ is a given nonnegative function.
Here $0<T\le\infty$ is the maximal time of existence; we denote $Q_T=\mathbb{R}^N\times(0,T)$, $Q=Q_\infty$. The fractional time derivative is defined, for smooth enough functions, by
\begin{equation}\label{Caputo}
\mathcal{D}_t f(t)=\partial_t(\kappa\star(f-f(0)))(t)=\partial_t\left(\int_{0}^t(f(\tau)-f(0))\kappa(t-\tau)\,d\tau\right),
\end{equation}
where the kernel $\kappa$ satisfies
\begin{equation}\label{K0}\tag{K0}
\begin{array}{l}
\kappa\in L^1_{loc}(\mathbb{R}^+), \;\text{nonnegative, nonincreasing},\;\lim\limits_{t\to\infty}\kappa(t)=0, \\ [2mm]
\exists\,\ell\in L^1_{loc}(\mathbb{R}^+), \;\text{nonnegative, nonincreasing with }\;\ell\star\kappa=1.
\end{array}\end{equation}
The symbol $\star$ denotes convolution in time (the functions involved are assumed to vanish for negative times). The existence of the so called \emph{conjugate kernel} $\ell$ is easy to derive, see Section~\ref{subsec-Zt}, but the monotonicity property, which we do not know if it is true in general, is crucial in our arguments.

As to the diffusion operator $\mathcal{L}_x$, it is a nonlocal L\'evy operator defined, for smooth enough functions, by
\begin{equation}\label{operator}
\mathcal{L}_xw(x)=\text{P.V.}\int_{\mathbb{R}^N}(w(x)-w(z))\mathcal{J}(x-z)\,dz,
\end{equation}
where
\begin{equation}\label{J0}\tag{J0}
\begin{array}{c}
\displaystyle\mathcal{J} \;\text{nonnegative, radially symmetric, nonincreasing near infinity}, \; \\ [3mm]
\displaystyle\int_{\mathbb{R}^N}\min\{1,|z|^2\}\mathcal{J}(z)\,dz<\infty,
\end{array}
\end{equation}
and P.V. stands for principal value.

Recently space-time fractional differential equations have been used in lots of applications, such as memory effects,  long-range interactions, anomalous diffusion, quantum mechanics, L\'evy flights, see for instance \cite{Cartea07,delCastillo05,Ezzat11,Metzler00,Zaslavsky02}.

The special power case in the time operator, $\kappa(t)=t^{-\alpha}$ is (a multiple of) the Caputo derivative
\begin{equation}\label{Caputo2}
\partial_t^\alpha f(t)={}_{RL}\partial_t^\alpha (f(t)-f(0))=\dfrac1{\Gamma(1-\alpha)}\partial_t\left(\int_{0}^t\frac{f(\tau)-f(0)}{(t-\tau)^\alpha}\,d\tau\right),
\end{equation}
the standard Riemann-Liouville fractional derivative of order $\alpha\in(0,1)$ of $f(t)-f(0)$.

We want to consider as much general operators as possible. In particular, the basic theory for our equation can be established with the only assumptions on $\kappa$ made above. Nevertheless, in order to characterize the blow-up phenomenon in terms of the power $p$, some homogeneity on the operators must be assumed. We then impose for some results the hypothesis that the time operator behaves like the Caputo fractional derivative, i.e., satisfying,
\begin{equation}
  \label{K1}
  \tag{K1}
c_1t^{-\alpha}\le \kappa(t)\le c_2t^{-\alpha},\qquad t>0, \quad0<\alpha<1,
\end{equation}
or even that $\kappa(t)=\frac1{\Gamma(1-\alpha)}t^{-\alpha}$, i.e., $\mathcal{D}_t=\partial_t^\alpha$, for some other results.

Also, the special case in the spatial operator where $\mathcal{J}(z)=|z|^{-N-\beta}$, $0<\beta<2$, reduces to (a multiple of) the fractional Laplacian,
\begin{equation}\label{fraclap}
(-\Delta)^{\beta/2}w(x)=c_{N,\beta}\int_{\mathbb{R}^N}\frac{w(x)-w(z)}{|x-z|^{N+\beta}}\,dz.
\end{equation}
The normalization constant $c_{N,\beta}$ 
is chosen so that
$$
\left[(-\Delta)^{\beta/2} w\right]\widehat{\,}\,(\xi)= |\xi|^{\beta}\widehat{w}(\xi),
$$
Here $\;\widehat{}\;$ stands for Fourier transform.

On the other hand, if the kernel $\mathcal{J}$ is integrable and for instance $\|\mathcal{J}\|_1=1$, then
\begin{equation}\label{julio}
\mathcal{L}_x w(x)=w(x)-w*\mathcal{J}(x)=w(x)-\int_{\mathbb{R}^N}w(y)\mathcal{J}(x-y)\,dy,
\end{equation}
and
$$
\widehat{\mathcal{L}_x w}(\xi)=(1-\widehat{\mathcal{J}}(\xi))\widehat{w}(\xi).
$$
Here $*$ means convolution in space.

These two operators \eqref{fraclap} and \eqref{julio} are the prototype of two completely different families of nonlocal operators, and their study has been usually been considered separately. We want to consider
nonlocal operators similar to the fractional Laplacian, but also allowing, for some results, operators with kernels on the integrable side. And also kernels with compact support have been considered. This is reflected in the hypotheses concerning the behaviour of the L\'evy kernel $\mathcal{J}$ at zero or at infinity, which determines the behaviour of the symbol $m(\xi)$,
\begin{equation}
  \label{symbol}
  \widehat{\mathcal{L}_x w}(\xi)=m(\xi)\widehat{w}(\xi).
\end{equation}
The behaviour of this symbol is characterized in an Appendix in terms of the kernel $\mathcal{J}$. In particular we consider the following different assumptions:
\begin{equation}
  \label{J1}
  \tag{J1}
  \mathcal{J}(z)\le  c|z|^{-N-\gamma}\qquad\text{for } |z|>1,\quad 0<\gamma\le\infty.
\end{equation}
\begin{equation}
  \label{J2}
  \tag{J2}
  \mathcal{J}(z)\ge  c|z|^{-N-\omega}\qquad\text{for } |z|>1,\quad 0<\omega\le\infty.
\end{equation}
\begin{equation}
  \label{J3}
  \tag{J3}
  \mathcal{J}(z)\ge  c|z|^{-N-\beta}\qquad\text{for } 0<|z|<1,\quad 0<\beta<2.
\end{equation}
The constants $c>0$ in the above conditions are in general different. Hypotheses \eqref{J3} marks the differential character of the spatial diffusion operator, and departs form the nonsingular case $\mathcal{J}\in L^1(\mathbb{R}^N)$; both types of operators are allowed by \eqref{J1} and \eqref{J2}, which deal only with the tail of the kernel. Clearly $\omega\ge\gamma$. If $\mathcal{J}$ has finite second order momentum, which includes the case of compact support, we can take $\gamma=\omega=2$, as we will see below.

A particular case is
\begin{equation}\label{stable-like}
 c_1|z|^{-N-\beta}\le  \mathcal{J}(z)\le  c_2|z|^{-N-\beta}\qquad\text{for any } z\in\mathbb{R}^N\setminus\{0\},
\end{equation}
for some $0<\beta<2$, which implies that the operator behaves as the fractional Laplacian  of order $\beta$, $\mathcal{L}_x\sim(-\Delta)^{\beta/2}$, and it is called \emph{stable-like}.
More precise assumptions will be defined for each result. Hypotheses \eqref{K0} and \eqref{J0} are assumed throughout the paper without further mention.

It can be proved that there exist two functions, $Z_t$ and $Y_t$, such that, if $v_0,f$ are regular enough,  the  function given by Duhamel's type formula
\begin{equation}\label{Duhamel}
v(x,t)=\int_{\mathbb{R}^N}Z_t(x-y)v_0(y)\,dy+\int_0^t\int_{\mathbb{R}^N}Y_{t-\tau}(x-y)f(y,\tau)\,dyd\tau,
\end{equation}
satisfies
\begin{equation}\label{prob-linear}
\left\{
\begin{array}{ll}
(\mathcal{D}_t +\mathcal{L}_x)v=f,\qquad & \text{in } Q,\\
v(x,0)=v_0(x).
\end{array}\right.
\end{equation}
The above formula can be justified using Fourier/Laplace transforms in space and time (denoted by $\;\widehat{}\;/\;\widetilde{}\;$). The kernel $Z_t$ is the fundamental solution of the double nonlocal heat operator $\mathcal{H}=\mathcal{D}_t +\mathcal{L}_x$, while $Y_t$, called the resolvent, is some time differential operator applied to $Z_t$, see \eqref{Z-Y}. The second term is a double convolution, in space and time, but to avoid confusion we only use the notation for time convolution in the proof of some abstract results in Sections~\ref{subsec-Zt}, \ref{subsec-Yt} and \ref{sec-linearproblem}. The pair $(Z_t, Y_t)$, called \emph{the matrix of fundamental solutions}, or \emph{fundamental pair}, of operator $\mathcal{H}$, is defined in a precise way in Section~\ref{sec-fundamentalair}, as well as studied its properties.

Whenever the two terms in formula \eqref{Duhamel} are well defined  we say that $v$ is a \emph{mild solution} to problem \eqref{prob-linear}. This allows to define the concept of mild solution to our problem~\eqref{eq.principal}, see Section~\ref{sec-basictheory}.

\

{\sc Previous results on blow-up}

The story of blow-up for parabolic equations started in the sixties of the previous century  with the study of the semilinear (local) heat equation
\begin{equation}\label{local-eq}
(\partial_t-\Delta) u=u^p.
\end{equation}
See the seminal works of Kaplan \cite{Kaplan63} and Fujita \cite{Fujita66}. For that equation there exist two critical exponents, the global existence exponent $p_0=1$ and the so called \emph{Fujita} exponent $p_*=1+2/N$, such that: $(i)$ all solutions are global in time if $p\le p_0$, $(ii)$ all solutions blow up in finite time if $p_0<p\le p_*$, and $(iii)$ there exist both, global in time solutions and blow-up solutions if $p>p_*$. See also~\cite{Hayakawa73,KobayashiSiraoTanaka77}.  We concentrate on the range $p>1$, since for $p<1$ there are nonuniqueness issues, though it is easy to check that all possible solutions are global in time; the case $p=1$ results on a linear equation, and solutions are also global.

From those results related to equation \eqref{local-eq} a lot of research has been performed for different local diffusion operators, like the $p$--Laplacian or the Porous Medium. See for instance the works~\cite{CazenaveDicksteinWeissler08,DengLevine00} and the review books \cite{Hu18,QuittnerSouplet07,SamarskiiGalaktionovKurdyumovMikhailov95}.

In the fractional derivatives framework we quote the works \cite{CortazarQuirosWolanski24,NagasawaSirao69,Sugitani75,ZhangSun15} for the space fractional, time fractional or even double fractional equation
\begin{equation}\label{double-eq}
(\partial_t^\alpha+(-\Delta)^{\beta/2}) u=u^p,
\end{equation}
$0<\alpha\le1$, $0<\beta\le2$.
Equation~\eqref{local-eq} then corresponds to $\alpha=\beta/2=1$. The papers \cite{NagasawaSirao69,Sugitani75} for one side and \cite{ZhangSun15} for the other consider, respectively, the cases $\alpha=1$ and $\beta=2$, showing that the Fujita exponent is $p_*=1+\beta/N$. In the case $0<\alpha<\beta/2=1$, it is also proved that for the critical exponent $p=1+2/N$ there exist global solutions, contrary to what happens for $\alpha=1$.

The inner case $0<\alpha<1$, $0<\beta<2$ has been recently studied, see the preprint \cite{CortazarQuirosWolanski24}, where the authors obtain the Fujita exponent $p_*=1+\beta/N$, in a more general context of blow-up in $L^q(\mathbb{R}^N)$, $1\le q\le\infty$. They also prove that the exponent $p_*$ belongs to the range where also global solutions exist.

Blow-up for integrable kernels has been studied in \cite{GarciaMelianQuiros10}, in the compactly supported case, and the Fujita exponent is $p^*=1+2/N$, and in~\cite{Alfaro17} for general kernels, where the Fujita exponent is characterized by the finiteness or not of the second moment of the kernel.

Related with the doubly nonlocal operator studied here, but of a different nature, we have studied in the recent paper \cite{FerreiradePablo24} the blow-up problem involving  a power of the heat operator
$$
\left\{
\begin{array}{ll}
(\partial_t-\Delta)^{\sigma} u=u^p,\qquad & (x,t)\in \mathbb R^N\times (0,T),\\
u(x,t)=f(x,t),\qquad & (x,t)\in \mathbb R^N\times (-\infty,0].
\end{array}\right.
$$
We proved that the Fujita exponent is given by $p_*=1+\frac{2\sigma}{N+2(1-\sigma)}$, which is different to the Fujita exponent for problem \eqref{double-eq} with $\alpha=\beta/2=\sigma$. Nevertheless, the above power operator is related to the Marchaud fractional time derivative more than Caputo derivative: all the memory data $f$ for $t<0$ is involved in the evolution.

In this paper we consider the generalization \eqref{eq.principal} to equation~\eqref{double-eq} by introducing the more general kernels \eqref{Caputo}, \eqref{operator}. We show that the existence of blow-up depends only on the tail of the kernel~$\mathcal{J}$, and only on an estimate from above, condition \eqref{J1}. On the contrary, the existence of global solutions depends on the estimate on the tail from below~\eqref{J2}, together with a minimum of singularity of the kernel at the origin, condition~\eqref{J3}, ensuring a smoothing effect. Observe that compactly supported kernels can be considered for both results. Condition~\eqref{K1} is required whenever precise estimates on the functions $Z_t$ and $Y_t$ are needed.

Put
$$
\overline\gamma=\min\{\gamma,2\},\qquad \varpi=\min\{\omega,2\}.
$$
Our main result is the following.
\begin{teo}\label{teo-BUP} \

\begin{itemize}
\item Assume condition \eqref{J1}. If $p>1$ there exist blowing-up solutions to problem \eqref{eq.principal}.
\item Assume also condition \eqref{K1}. If  $1<p< 1+\overline\gamma/N$  then every solution blows up in finite time.
  \item Assume conditions \eqref{K1}, \eqref{J2} and \eqref{J3} with $\beta>\frac{N\varpi}{N+\varpi}$ and $N<2\beta$.
 If $p\ge 1+\varpi/N$  then there exist global solutions.
\item If $\mathcal{D}_t=\partial_t^\alpha$ the result in the previous item holds true for any dimension $N\ge1$. \end{itemize}
\end{teo}

We do not know if the restriction $\beta>\frac{N\varpi}{N+\varpi}$ in the last two items is necessary or technical. Observe that if $\mathcal{J}(z)\sim|z|^{-N-\gamma}$ at infinity, $\gamma>0$, and the above restrictions on $\beta$ hold, the Fujita exponent is $p_*=1+\overline\gamma/N$. In particular for the stable-like case \eqref{stable-like} it is $p_*=1+\beta/N$; see \cite{CortazarQuirosWolanski24} for the double fractional equation \eqref{double-eq}.

\

{\sc Organization of the paper}

We have tried to make this work self-contained for the benefit of the reader. We start with  a long preliminary Section~\ref{sec-fundamentalair} where we  construct the fundamental pair $(Z_t, Y_t)$, and study its integrability properties, in the more general case and in particular situations. In Section~\ref{sec-basictheory} we consider problem~\eqref{eq.principal} and prove the existence and uniqueness of a mild solution. We also prove that mild solutions are very weak solutions. In Sections~\ref{sec-bup1} and~~\ref{sec-bup2} we characterize the ranges for which there exist blow-up solutions and/or global solutions. We end with an Appendix~\ref{app-A} where we study the properties of the symbol $m(\xi)$ that characterizes the diffusion operator $\mathcal{L}_x$ and its associated heat kernel.

The letters $c,c_1,c_2,\dots$ denote generic constants non depending on the relevant quantities, and that may change through the calculations.

\

\section{The fundamental pair}\label{sec-fundamentalair}

\setcounter{equation}{0}

In this section we define the concept of mild solution to the linear problem associated to our operator $\mathcal{H}=\mathcal{D}_t +\mathcal{L}_x$, that we write here again for convenience.
\begin{equation}\label{prob-linear2}
\left\{
\begin{array}{ll}
\mathcal{H}v=f,\qquad & \text{in } Q_T,\\
v(x,0)=v_0(x).
\end{array}\right.
\end{equation}
We denote $X=L^1(\mathbb{R}^N)\cap L^\infty(\mathbb{R}^N)$.
\begin{defi}
  Given the data $u_0\in X$ and $f\in L^\infty((0,T);X)$, a function $u\in C([0,T);X)$ is said to be a mild solution of problem \eqref{prob-linear2} if
\begin{equation}\label{mild0}
u(x,t)=Z_t*u_0(x)+\int_0^tY_{t-\tau}*f(x,\tau)\,d\tau,
\end{equation}
a.e $x\in\mathbb{R}^N$ and for all $0<t<T$.\end{defi}

We must give a sense to the above formula by constructing the functions $Z_t,Y_t$, so that $u$  is well defined and satisfies $u\in C([0,T);X)$. See \cite{Duan05,KemppainenSiljanderZacher17} for the double fractional case.

\subsection{The kernel $Z_t$.}\label{subsec-Zt}
First we take $Z_t$ as the fundamental solution of the operator $\mathcal{H}$, that is
\begin{equation}\label{zeta}
\left\{
\begin{array}{ll}
\mathcal{H} Z_t=0,\qquad & \text{in } Q,\\
Z_0=\delta(x).
\end{array}\right.
\end{equation}
In that way, if $v_0$ is regular enough, then the function $v=Z_t *v_0$ satisfies
\begin{equation}\label{linearzero}
\left\{
\begin{array}{ll}
\mathcal{H} v=0,\qquad & \text{in } Q,\\
v(x,0)=v_0(x).
\end{array}\right.
\end{equation}
In the special double fractional case the kernel $Z_t$ is self-similar
\begin{equation}\label{selfsimilarZ}
Z_t(x)=t^{-\frac{\alpha N}{\beta}}F(xt^{-\frac{\alpha}{\beta}}),\qquad \widehat{F}(\xi)=E_\alpha(-|\xi|^{\beta}),
\end{equation}
where $E_a(r)=\sum\limits_{j=0}^\infty\frac{r^j}{\Gamma(a j+1)}$ is the Mittag-Leffler function. From this formula in \cite{Duan05} it is obtained an explicit representation of $F$ in terms of Fox functions, and thus deduced its behaviour (and then that of $Z_t$). In particular $F>0$ in $\mathbb{R}^N$ and it has a singularity at the origin depending on the dimension and $\beta$. It is proved  in \cite{Duan05} that $F\in L^q(\mathbb{R}^N)$ if and only if $1\le q<\frac N{(N-\beta)_+}$.

For the general nonlocal operator $\mathcal{H}$ considered here there is no such an explicit representation so we must study the properties of $Z_t$ in an indirect way from the properties of its Fourier/Laplace transform.

In order to decribe $Z_t$ we apply the Fourier transform in $x$ to the above problem, getting the nonlocal problem in $t$
\begin{equation}\label{Z-Fourier}
(\mathcal{D}_t+m(\xi))\widehat{Z}_t(\xi)=0, \quad t>0,\qquad \widehat{Z}_0(\xi)=1,
\end{equation}
where $m(\xi)$ is the symbol of $\mathcal{L}_x$. Then $\widehat{Z}_t(\xi)=\varrho_1(t;m(\xi))$, where $\varrho_1=\varrho_1(t;\mu)$ solves
\begin{equation}\label{eq:ro1}
(\mathcal{D}_t+\mu) \varrho_1=0, \quad t>0,\qquad \varrho_1(0;\mu)=1.
\end{equation}

\

We now observe that for any given $\kappa\in L^1_{loc}(\mathbb{R}^+)$, nonnegative and nonincreasing vanishing at infinity, there exists a conjugate kernel $\ell$ such that
\begin{equation}\label{conjugate}
  k\star \ell(t)=\int_0^t\kappa(\tau)\ell(t-\tau)\,d\tau=1.
\end{equation}
The argument is as follows, see \cite{Pruss93}, Propositions 4.2 and 4.3: Since the Laplace transform $\widetilde\kappa$ is a completely monotone function, then $\varphi(s)=s\widetilde\kappa(s)$ is a Bernstein function, and $1/\varphi(s)$ is also a completely monotone function; therefore there exists a nonnegative  function $\ell$ such that $\widetilde\ell(s)=1/\varphi(s)$, which implies $\widetilde k(s)\widetilde\ell(s)=1/s$; uniqueness of the Laplace transform gives \eqref{conjugate}. However, the monotonicity of $\ell$ is crucial in what follows, so we must impose this property in assumption~\eqref{K0}.

For example, in the Caputo derivative case $\kappa(t)=\frac1{\Gamma(1-\alpha)}t^{-\alpha}$, we get $\ell(t)=\frac1{\Gamma(\alpha)}t^{\alpha-1}$. See for instance~\cite{VergaraZacher15} for more explicit examples.

For the general case we define
\begin{equation}\label{haches}
h_\kappa(t)=\int_0^t\kappa(\tau)\,d\tau,\qquad h_\ell(t)=\int_0^t\ell(\tau)\,d\tau.
\end{equation}

\begin{prop}
\begin{equation}\label{ell}
    \ell(t)\le \frac1{h_\kappa(t)},\qquad h_\ell(t)\ge \frac t{h_\kappa(t)}.
  \end{equation}
  In particular, assuming \eqref{K1} we have
  \begin{equation}\label{ell2}
    \ell(t)\le ct^{\alpha-1},\qquad h_\ell(t)\ge ct^\alpha.
  \end{equation}
\end{prop}

\begin{proof}
We first have
$$
1=\int_0^t\kappa(\tau)\ell(t-\tau)\,d\tau\ge \ell(t)\int_0^t\kappa(\tau)\,d\tau=\ell(t)h_\kappa(t).
$$
In the other direction
$$
\begin{array}{rl}
1&\displaystyle=\int_0^{2t}\kappa(\tau)\ell(2t-\tau)\,d\tau\le  \ell(t)\int_0^t\kappa(\tau)\,d\tau+\kappa(t)\int_t^{2t}\ell(2t-\tau)\,d\tau\\ [4mm]
&\displaystyle=\ell(t)h_\kappa(t)+\kappa(t)h_\ell(t)=(h_\ell h_\kappa)'(t).
\end{array}
$$
Thus $h_\ell(t)h_\kappa(t)\ge t$.

\end{proof}

In terms of the conjugate kernel $\ell$ equation \eqref{eq:ro1} is written as the linear Volterra equation
\begin{equation}\label{eq:ro1-b}
\varrho_1+\mu (\ell\star\varrho_1)=1.
\end{equation}
In \cite{Miller71} it is proved that equation \eqref{eq:ro1-b} has a solution $\varrho_1\in W_{loc}^{1,1}(\mathbb{R}^+)$, $0<\varrho_1<1$, $\partial_t\varrho_1\le0$ for every $\mu>0$, see also \cite{VergaraZacher15}.
It is characterized by its Laplace transform (in $t$)
\begin{equation}\label{laplace-rho}
  \widetilde{\varrho}_1(s;\mu)=\frac{1/s}{1+\mu\widetilde{\ell}(s)}=
  \frac{\widetilde{\kappa}(s)}{s\widetilde{\kappa}(s)+\mu}.
\end{equation}
We can also write
$$
\widetilde{\varrho}_1(s;\mu)=\frac1s\sum_{n=0}^\infty(-\mu\widetilde{\ell}(s))^n
=\frac1s\left(1-\mu\sum_{n=0}^\infty(-\mu)^n(\ell\stackrel{n)}\star\ell)\,\widetilde{}\,(s)\right).
$$
Here $\ell\stackrel{n)}\star\ell$ denotes the $n$--times convolution, so that $\ell\stackrel{0)}\star\ell=\ell$ and  $(\ell\stackrel{n)}\star\ell)\,\widetilde{}=(\widetilde{\ell}(s))^{n+1}$. Therefore
\begin{equation}\label{rho1-sum}
\varrho_1(t;\mu)=1-\mu\int_0^t\sum_{n=0}^\infty(-\mu)^n(\ell\stackrel{n)}\star\ell)(\tau)\,d\tau.
\end{equation}

Using any of the above formulas in the particular Caputo case $\mathcal{D}_t=\partial_t^\alpha$ it is easy to recover the expression of $\varrho_1$ as a Mittag-Leffler function.

Observe also that $\ell=1$ in \eqref{eq:ro1-b}, which corresponds to $\mathcal{D}_t=\partial_t$, gives $\varrho_1(t;\mu)=e^{-\mu t}$, and thus $G_t=(e^{-m(\xi)t})^\vee$, which is the Gauss kernel of the operator $\partial_t+\mathcal{L}_x$.

\

Now, if $z_0\in L^2(\mathbb{R}^N)$, the function $a(\xi,t)=\varrho_1(t;m(\xi))z_0(\xi)$ is a solution in $C((0,T);L^2(\mathbb{R}^N))$, for every $T>0$, to the problem
$$
\left\{
\begin{array}{ll}
(\mathcal{D}_t +m(\xi))a=0,\qquad & \text{in } Q_T,\\
a(\xi,0)=z_0(\xi).
\end{array}\right.
$$
Therefore, if $v_0\in L^2(\mathbb{R}^N)$, the function $v=Z_t*u_0\in C((0,T);L^2(\mathbb{R}^N))$ is a solution  for every $T>0$ to problem \eqref{linearzero}.

\subsection{The kernel $Y_t$.}\label{subsec-Yt} We now take a look at the second term in \eqref{mild0}. We want to solve the nonhomogeneous linear problem
\begin{equation}\label{linear-nonh}
\left\{
\begin{array}{ll}
\mathcal{H} w=f,\qquad &\text{in } Q,\\
w(x,0)=0,
\end{array}\right.
\end{equation}
with a given smooth function $f$. Taking Fourier transform and inverting the time operator, we get the expression
$$
\widehat w+m(\ell\star\widehat w)=\ell\star\widehat f.
$$
Thus, in the same way as before we consider now the solution $\varrho_2=\varrho_2(t;\mu)$ to the Volterra equation
\begin{equation}\label{eq:ro2}
\varrho_2+\mu (\ell\star\varrho_2)=\ell, \quad t>0.
\end{equation}
Then, if $g\in C((0,T);L^2(\mathbb{R}^N))$, the function $b(\xi,t)=\varrho_2(\cdot;m(\xi))\star g(t)$ is a solution in $C((0,T);L^2(\mathbb{R}^N))$, for every $T>0$, to the problem
\begin{equation}\label{Y-Fourier}
\left\{
\begin{array}{ll}
(\mathcal{D}_t +m(\xi))b=g,\qquad & \text{in } Q_T,\\
b(\xi,0)=0.
\end{array}\right.
\end{equation}
Finally, defining $\widehat{Y}_t(\xi)=\varrho_2(t;m(\xi))$,  given  $f\in C((0,T);L^2(\mathbb{R}^N))$ the function $w=\int_0^tY_{t-s}*f(s)\,ds$ solves problem \eqref{linear-nonh}.
See \cite{VergaraZacher15}.

The above resolvent $\varrho_2$ exists for the family of kernels considered, i.e., equation \eqref{eq:ro2} has a unique solution, which is locally integrable and nonnegative, see \cite[Theorem 2.3.1]{GripenbergLondenStaffans90}.
It  is constructed recursively
\begin{equation}\label{recurs-rho2}
\rho_2=\lim_{n\to\infty}r_n,\qquad r_{n+1}=\ell-\mu(\ell\star r_n),\;r_1=\ell,
\end{equation}
and is characterized by its Laplace transform
\begin{equation}\label{laplace-rho2}
  \widetilde{\varrho}_2(s;\mu)=\frac{\widetilde{\ell}(s)}{1+\mu\widetilde{\ell}(s)}=\frac{1}{s\widetilde{\kappa}(s)+\mu}.
\end{equation}
Comparing with \eqref{laplace-rho}, we get the relation
\begin{equation}\label{rho1-2}
\widetilde{\varrho}_1=\widetilde{\kappa}\widetilde{\varrho}_2\;\Rightarrow\;\varrho_1=\kappa\star\varrho_2\;\Rightarrow\;
\ell\star\varrho_1=1\star\varrho_2\;\Rightarrow\;\varrho_2=\partial_t(\ell\star\varrho_1).
\end{equation}
This gives
$$
\widehat{Y}_t(\xi)=\partial_t(\ell\star\widehat{Z}_t(\xi)),
$$
and as a corollary,
\begin{prop}
The fundamental pair $(Z_t,Y_t)$ satisfies the following relation
\begin{equation}\label{Z-Y}
Y_t=\partial_t(\ell\star Z_t).
\end{equation}
\end{prop}

On the other hand, using formula \eqref{rho1-sum} we have
\begin{equation}\label{rho2-sum}
\varrho_2(t;\mu)=-\frac1\mu\partial_t\varrho_1(t;\mu)=\sum_{n=0}^\infty(-\mu)^{n}(\ell\stackrel{n)}\star\ell)(t),
\end{equation}
see also \eqref{recurs-rho2}.

In the Caputo case the above formulas give the explicit expression
\begin{equation}\label{ygriega-mitag}
\begin{array}{rl}
\varrho_2(t;\mu)&\displaystyle=\sum_{n=1}^\infty(-\mu)^{n-1}\frac{t^{n\alpha-1}}{\Gamma(n\alpha)}=t^{\alpha-1}\sum_{j=0}^\infty\frac{(-\mu t^\alpha)^{j}}{\Gamma((j+1)\alpha)} \\ [3mm]
&\displaystyle=t^{\alpha-1}E_{\alpha,\alpha}(-\mu t^\alpha),
\end{array}
\end{equation}
where $E_{a,b}(r)=\sum\limits_{j=0}^\infty\frac{r^j}{\Gamma(a j+b)}$ is the two-parametric Mittag-Leffler function. Also, formula \eqref{Z-Y} becomes here,
\begin{equation}\label{Zt-Tt}
Y_t=\partial_t(\ell\star Z_t)={}_{RL}\partial_t^{1-\alpha}(Z_t).
\end{equation}
This formula can be checked directly by using the Riemann-Liouville fractional derivative of a power in the expression of the Mittag-Leffler function.

We end this section by observing that  $Z_t$ and $Y_t$ are nonnegative. A direct proof could be performed by a comparison argument. Nevertheless, as we will see, it is an easy consequence of the subordination formula given in Section~\ref{subordination}.

As a corollary we have
\begin{equation}\label{norm1}
\begin{array}{c}
\|Z_t\|_1=\widehat{Z}_t(0)=\varrho_1(t;0)=1,\\ [4mm]
\|Y_t\|_1=\widehat{Y}_t(0)=\varrho_2(t;0)=\ell(t).
\end{array}\end{equation}
Moreover by \eqref{ell} (and \eqref{ell2} if we assume \eqref{K1}), if $g\in L^q(\mathbb{R}^N)$ for some $1\le q\le\infty$,
\begin{equation}\label{eq:u0zyt-inf}
\|Z_t*g\|_q\le \|g\|_q,\qquad\|Y_t*g\|_q\le \ell(t)\|g\|_q\le ct^{\alpha-1}\|g\|_q,
\end{equation}
for every $t>0$.

\subsection{Estimates for the fundamental pair: The Fourier transform strategy}\label{sec-Fourier}

We have defined $Z_t(x)$ and $Y_t(x)$ as the inverse Fourier transforms of $\varrho_1(t;m(\xi))$ and $\varrho_2(t;m(\xi))$, respectively. We now try to estimate these two last functions and apply the inverse Fourier transform. Unfortunately this restricts the dimension since $Z_t,Y_t\notin L^\infty(\mathbb{R}^N)$. In the next section we follow another strategy in order to avoid this restriction; we say in advance that we are able to do that but only in the Caputo case $\mathcal{D}_t=\partial_t^\alpha$.

We start with $\varrho_1$, which can be estimated easily in the following way:
$$
1=\varrho_1+\mu (\ell\star\varrho_1)\ge\varrho_1+\mu (1\star\ell)\varrho_1,
$$
that is
\begin{equation}\label{bounds:s}
\varrho_1(t;\mu)\le\frac1{1+\mu(1\star\ell(t))}=\frac1{1+\mu h_\ell(t)}.
\end{equation}

This bound is used now to estimate the integrability properties of $\widehat Z_t$ and then of $Z_t$.
\begin{lema}\label{lem-Z-estim}
Assume hypotheses \eqref{J2} and \eqref{J3} with $N<2\beta$. Then $Z_t\in L^\eta(\mathbb{R}^N)$ for $t>0$, where $1\le \eta<\frac N{(N-\beta)_+}$ if $N\ge\beta$, or $1\le \eta \le\infty$ if $1=N<\beta$. Moreover
\begin{equation}\label{decayZt}
\|Z_t\|_\eta\le c\begin{cases}
  h_\ell(t)^{-\frac{ N}{\beta}(1-1/\eta)} &t\le1, \\ h_\ell(t)^{-\min\{1,\frac{ N}{\varpi}(1-1/\eta)\}}&t>1.
  \end{cases}
\end{equation}
If we also assume condition \eqref{K1}, the corresponding explicit estimates are deduced by using that $h_\ell(t)\ge ct^\alpha$.
\end{lema}

\begin{proof}
Recall that $\|Z_t\|_1=1$. Now we use \eqref{bounds:s}, that gives
\begin{equation}\label{bounds2:hatZ}
\widehat{Z}_t(\xi)\le\frac 1{1+cm(\xi)h_\ell(t)}.
\end{equation}
We calculate the $L^\sigma$ norm of $\widehat{Z}_t$ using the estimates \eqref{rho>1} and \eqref{rho<1} of the symbol $m(\xi)$ from Appendix~\ref{app-A},
$$
\int_{\mathbb{R}^N}|\widehat{Z}_t(\xi)|^\sigma\,d\xi\le c_1\int_0^1\frac{\rho^{N-1}}{(1+\rho^{\varpi}h_\ell(t))^\sigma}\,d\rho+
c_2\int_1^\infty\frac{\rho^{N-1}}{(1+\rho^{\beta}h_\ell(t))^\sigma}\,d\rho
=I_1+I_2.$$
The first integral is always finite while the second integral converges only if $\beta\sigma>N$.
 Now
$$
\begin{array}{l}
\displaystyle I_1=c_1h_\ell(t)^{-\frac{N}{\varpi}}\int_0^{h_\ell(t)}\frac{\zeta^{\frac N{\varpi}-1}}{(1+\zeta)^\sigma}\,d\zeta\le c_1\begin{cases}
  1 &t\le1, \\ h_\ell^{-\frac{N}{\varpi}}(t)+h_\ell^{-\sigma}(t)&t>1;
\end{cases} \\ [6mm]
\displaystyle I_2=c_2h_\ell^{-\frac{N}{\beta}}(t)\int_{h_\ell(t)}^\infty\frac{\zeta^{\frac N{\beta}-1}}{(1+\zeta)^\sigma}\,d\zeta\le c_2\begin{cases}
 h_\ell^{-\frac{N}{\beta}}(t)&t\le 1,\\
 h_\ell^{-\sigma}(t) &t>1. \end{cases}
\end{array}$$
This gives
$$
\int_{\mathbb{R}^N}|\widehat{Z}_t(\xi)|^\sigma\,d\xi\le c\begin{cases}
  h_\ell^{-\frac{N}{\beta}}(t) &t\le1, \\ h_\ell^{-\min\{\sigma,\frac{N}{\varpi}\}}(t)&t>1.
  \end{cases}
$$
We now have
$$
\|Z_t\|_\eta\le c\|\widehat{Z}_t\|_\sigma\qquad \max\{N/\beta,1\}<\sigma=\frac \eta{\eta-1}\le2.
$$
We get the desired estimate for $2\le \eta<\frac{N}{(N-\beta)_+}$. We end with an interpolation argument for $1<\eta<2$. We can also include the exponent $\eta=\infty$ if $N<\beta$.
\end{proof}

\begin{cor}\label{cor-Zt}
In the above hypotheses, if $g\in L^r(\mathbb{R}^N)$ for some $1\le r\le\infty$,
\begin{equation}\label{eq:u0zt}
\|Z_t*g\|_q\le c h_\ell^{-\delta}(t)\|g\|_r
\end{equation}
for every $r\le q<\frac{Nr}{(N-\beta r)_+}$, if $r\le N/\beta$, for every $r\le q\le \infty$ if $r>N/\beta$,
where
\begin{equation}\label{exp:u0zt}
\delta=\delta(r,q;t)=\begin{cases}
  \frac{ N}\beta(1/r-1/q) &t\le1, \\ \min\{1,\frac{N}\varpi(1/r-1/q)\}&t>1.
\end{cases}
\end{equation}\end{cor}
Observe that assuming \eqref{K1} and putting $\omega=\beta\in(0,2)$ in \eqref{J2} and \eqref{J3}, then
we get the global estimate
\begin{equation}\label{Zt-eta}
\|Z_t*g\|_q\le c t^{-\frac{\alpha N}\beta(1/r-1/q)}\|g\|_r\qquad t>0,
\end{equation}
with the above restrictions on $q$ and $r$.

Now we turn our attention to the function $\varrho_2$ and the kernel $Y_t$. Using the relation \eqref{rho1-2} (omiting the $\mu$-dependence), and since both, $\varrho_1$ and $\ell$ are nonnegative and nonincreasing,
$$
\begin{array}{rl}
\varrho_2(t)&\displaystyle=\partial_t(\ell\star\varrho_1)(t)=\ell(t)\varrho_1(0)+\int_0^t
\ell(s)\partial_t\varrho_1(t-s)\,ds\\ [2mm]
&\displaystyle\le \ell(t)+\ell(t)\int_0^t\partial_t\varrho_1(t-s)\,ds=\ell(t)\varrho_1(t).\end{array}$$


\begin{cor}\label{cor-Yt}
In the hypotheses and notation of Corollary~\ref{cor-Zt},
\begin{equation}\label{eq:u0yt}
\|Y_t*g\|_q\le \ell(t)h_\ell^{-\delta}(t)\|g\|_r.
\end{equation}
that assuming also \eqref{K1} becomes
\begin{equation}\label{eq:u0yt-alpfa}
\|Y_t*g\|_q\le  ct^{\alpha-1-\alpha\delta}\|g\|_r.
\end{equation}
\end{cor}

\begin{proof}
  Just use the above estimate $\varrho_2\le \ell\varrho_1$ in the Fourier transform estimate, together with \eqref{ell2} and \eqref{decayZt},
  $$
  \|Y_t*g\|_q\le\|Y_t\|_\eta\|g\|_r\le c\|\widehat{Y}_t\|_\sigma\|g\|_r\le c\ell(t)\|\widehat{Z}_t\|_\sigma\|g\|_r\le c\ell(t)h_\ell^{-\delta}(t)\|g\|_r,
  $$
  where
$$
1+\frac1q=\frac1\eta+\frac1r,\qquad \frac1\eta+\frac1\sigma=1.
$$
\end{proof}

\subsection{Estimates for the fundamental pair: The subordination formula}\label{subordination}

In order to get rid of the restriction $N<2\beta$ we now show that the fundamental pair $(Z_t,Y_t)$ satisfies another more explicit expression using a subordination principle in terms of the nonlocal heat kernel $G_t$ with standard time derivative. This kernel is studied in Appendix~\ref{app-A}.

We want to define a function $\Psi_1(t,\tau)$ such that its Laplace transform in the second variable satisfies ${\widetilde \Psi_1}^{{}^\tau}(t,\mu)=\varrho_1(t;\mu)$, where $\varrho_1$ is the solution to \eqref{eq:ro1-b}, but we do not know if it exists. Instead, following \cite{Pruss93} we define $\Psi_1$ by its Laplace transform in $t$,
\begin{equation}\label{laplaceW}
  {\widetilde \Psi_1}^{{}^t}(s,\tau)=\widetilde \kappa(s)e^{-\tau s\widetilde \kappa(s)}.
\end{equation}
The existence of the so called \emph{propagation function} $\Psi_1$ follows from the properties of the function $\kappa$, since they imply that the function on the right is completely monotone, \cite[Proposition 4.5]{Pruss93}. Moreover $\Psi_1=\partial_\tau U$, where $U(t,\tau)\ge0$, and for each $\tau>0$, $U(\cdot,\tau)$ is  nonincreasing, while for $t>0$, $U(t,\cdot)$ is nondecreasing  and left-continuous, with
$$
U(0,\tau)=1,\quad \lim_{t\to\infty}U(t,\tau)=0,\quad U(t,0)=0,\quad \lim_{\tau\to\infty}U(t,\tau)=1.
$$
Therefore $\Psi_1\ge0$ and
$$
\int_0^\infty \Psi_1(t,\tau)\,d\tau=U(t,\infty)=1,
$$
so the propagation function is a probability density in $\tau$ for each $t>0$.

With this propagation function we can prove the following \emph{subordination formula}
\begin{teo}\label{prop-W}
The kernel $Z_t$ is given by
\begin{equation}\label{subord-W}
Z_t(x)=\int_0^\infty G_\tau(x)\Psi_1(t,\tau)\,d\tau.
\end{equation}
\end{teo}

\begin{proof} We first obtain the desired representation of the function $\varrho_1$ in terms of the propagation function $\Psi_1$. To do that we define the function
$$
\varphi(t)=\int_0^\infty e^{-\mu\tau}\Psi_1(t,\tau)\,d\tau.
$$
Its Laplace transform is
$$
\begin{array}{rl}
\displaystyle\widetilde\varphi(s)&\displaystyle=\int_0^\infty e^{-st}\int_0^\infty e^{-\mu\tau}\Psi_1(t,\tau)\,d\tau dt \\ [3mm]
&\displaystyle=\int_0^\infty e^{-\mu \tau} \int_0^\infty e^{-st}\Psi_1(t,\tau)\,dt d\tau \\ [3mm]
&\displaystyle=\int_0^\infty e^{-\mu \tau} \widetilde \kappa(s)e^{-\tau s\widetilde \kappa(s)}\, d\tau \\ [3mm]
&\displaystyle
=\frac{\widetilde{\kappa}(s)}{\mu +s\widetilde{\kappa}(s)}=\widetilde{\varrho}_1(s;\mu).
\end{array}
$$
Therefore
\begin{equation}\label{laplace-tau-W}
\varrho_1(t;\mu)=\varphi(t)={\widetilde \Psi_1}^{{}^\tau}(t,\mu),
\end{equation}
as we wanted.

Now we calculate the Fourier transform of the right-hand side in the expression \eqref{subord-W}.
$$
\int_0^\infty \widehat G_\tau(\xi)\Psi_1(t,\tau)\,d\tau=\int_0^\infty e^{-m(\xi)\tau}\Psi_1(t,\tau)\,d\tau=\varrho_1(t;m(\xi))=\widehat Z_t(\xi).
$$
\end{proof}

Observe that the Laplace bitransform of the propagation function is
\begin{equation}\label{bilaplace}
{\widetilde \Psi_1}^{{}^{t,\tau}}(s,\mu)=\frac{\widetilde{\kappa}(s)}{\mu +s\widetilde{\kappa}(s)}.
\end{equation}

Analogously,

\begin{teo}\label{prop-W}
The kernel $Y_t$ is given by the subordination formula
\begin{equation}\label{subord-W}
Y_t(x)=\int_0^\infty G_\tau(x)\Psi_2(t,\tau)\,d\tau,
\end{equation}
where
\begin{equation}\label{laplaceW2}
  {\widetilde \Psi_2}^{{}^t}(s,\tau)= e^{-\tau s\widetilde \kappa(s)}.
\end{equation}
\end{teo}

On the other hand since both, $\Psi_1$ and $\Psi_2$, as well as the heat kernel $G_t$, are nonnegative, we deduce the following result.

\begin{cor}
  The fundamental pair satisfies $Z_t,Y_t>0$.
\end{cor}

Now we try to introduce the estimates on the kernel $G_t$ from Proposition~\ref{prop-Gt} into the subordination formula. What we get is
$$
\|Z_t\|_1=\int_{\mathbb{R}^N}\int_0^\infty G_\tau(x)\Psi_1(t,\tau)\,d\tau dx=\int_0^\infty \Psi_1(t,\tau)\,d\tau=1,
$$
again, but for $q>1$,
$$
\begin{array}{rl}
\|Z_t\|_q&\displaystyle\le\int_0^\infty\|G_\tau\|_q \Psi_1(t,\tau)\,d\tau \\ [3mm]
&\displaystyle\le \int_0^1 \tau^{-\frac N\beta(1-1/q)} \Psi_1(t,\tau)\,d\tau
+\int_1^\infty \tau^{-\frac N\varpi(1-1/q)} \Psi_1(t,\tau)\,d\tau.
\end{array}
$$
This goes no further without more information on $\Psi_1$ from either formula \eqref{laplaceW} or \eqref{bilaplace}. This would require some a priori monotonicity properties of $\Psi_1$, that are not known in the general case, and the use of the Tauberian theory. This theory only provides with a behaviour of the function $U$, not $\Psi_1$, and this is not enough to get good estimates. We thus restrict ourselves to the Caputo fractional derivative, $\mathcal{D}_t=\partial_t^\alpha$.
Recall that in that case
\begin{equation}\label{psi1}
{\widetilde \Psi_1}^{{}^{\tau}}(t,\mu)=\varrho_1(t,\mu)=E_{\alpha}(-\mu t^\alpha),
\end{equation}
and it is well known the relation between Mittag-Leffler functions and Wright functions $W_{a,b}(z)=\sum\limits_0^\infty\frac{z^n}{n!\Gamma(an+b)}$ through the Laplace transform,  \cite{Mainardi94}:
\begin{equation}\label{W-phi}
(W_{-a,b})\widetilde{\hspace{2mm}}(-z)=E_{a,a+b}(-z),\qquad 0<a<1,\;b\in\mathbb{R},\;z>0.
\end{equation}
In particular from \eqref{psi1} we have that $\Psi_1$ is self-similar and is given by,
\begin{equation}\label{Psi-phi}
\Psi_1(t,\tau)=t^{-\alpha}\phi(\tau t^{-\alpha}),\qquad\phi(z)=W_{-\alpha,1-\alpha}(-z).
\end{equation}
An analogous procedure is done for the kernel $Y_t$ and the secondary propagation function. Since
\begin{equation}\label{psi2}
{\widetilde \Psi_2}^{{}^{\tau}}(t,\mu)=\varrho_2(t,\mu)=t^{\alpha-1}E_{\alpha,\alpha}(-\mu t^\alpha),
\end{equation}
we get from \eqref{W-phi}
\begin{equation}\label{W2}
\begin{array}{rl}
\Psi_2(t,\tau)&\displaystyle=t^{-1}W_{-\alpha,0}(-\tau t^{-\alpha})=t^{-1}\sum_{n=0}^\infty\frac{(-\tau t^{-\alpha})^n}{n!\Gamma(-\alpha n)} \\ [2mm] &\displaystyle=-\tau t^{-\alpha-1}\sum_{m=0}^\infty\frac{(-\tau t^{-\alpha})^m}{(m+1)!\Gamma(-\alpha m-\alpha)}\\ [2mm] &\displaystyle=\tau t^{-\alpha-1}\sum_{m=0}^\infty\frac{(-\tau t^{-\alpha})^m(\alpha m+\alpha)}{(m+1)!\Gamma(-\alpha m-\alpha+1)}\\ [3mm] &\displaystyle=\alpha \tau t^{-\alpha-1}W_{-\alpha,1-\alpha}(-\tau t^{-\alpha})=\alpha \tau t^{-\alpha-1}\phi(\tau t^{-\alpha}).
\end{array}
\end{equation}
In summary the subordination formulas read in the Caputo case
\begin{prop}\label{prop-W-caputo}
Let $\mathcal{D}_t=\partial_t^\alpha$. The kernel $Z_t$ and $Y_t$ are given by
$$
\begin{array}{l}
\displaystyle Z_t(x)=\int_0^\infty G_\tau(x)t^{-\alpha}\phi(\tau t^{-\alpha})\,d\tau=
\int_0^\infty G_{t^\alpha\theta}(x)\phi(\theta)\,d\theta, \\ [4mm]
\displaystyle Y_t(x)=\int_0^\infty G_\tau(x)\alpha t^{-\alpha-1}\tau\phi(\tau t^{-\alpha})\,d\tau
=t^{\alpha-1}\int_0^\infty G_{t^\alpha\theta}(x)\alpha\theta\phi(\theta)\,d\theta,
\end{array}
$$
where $\phi(z)=W_{-\alpha,1-\alpha}(-z)$.
\end{prop}
See also \cite{Bazhlekova00}.
In summary we can estimate the $L^q$ norms of $Z_t$ and $Y_t$ using the subordination formulas.
\begin{lema}\label{lem-Z-estim0}
Let $\mathcal{D}_t=\partial_t^\alpha$. Assuming hypotheses \eqref{J2} and \eqref{J3} it holds $Z_t\in L^\eta(\mathbb{R}^N)$ for $t>0$, where $1\le \eta<\frac N{(N-\beta)_+}$ if $N\ge\beta$, or $1\le \eta \le\infty$ if $N<\beta$. Moreover estimates~\eqref{decayZt} hold.
\end{lema}
\begin{proof}
Since $\Psi_1$ has integral 1 so has $\phi$. This gives
$$
\|Z_t\|_\eta\le\int_0^\infty \|G_{t^\alpha\theta}\|_\eta\phi(\theta)\,d\theta=\int_0^{t^{-\alpha}}+\int_{t^{-\alpha}}^\infty=I_1+I_2.
$$
We now introduce the estimates \eqref{estim-Gt}. We also need the behaviour of the Wright function at infinity, see again \cite{Mainardi94}:
$$
|W_{-\alpha,1-\alpha}(-z)|\le c_1z^{\frac1{1-\alpha^2}}e^{-c_2z^{\frac1{1-\alpha}}},\qquad z>1,\;c_i=c_i(\alpha)>0,
$$
and the value $\phi(0)=c>0$.
We calculate
$$
I_1\le  ct^{-\frac{N\alpha}\beta(1-1/\eta)}\int_0^{t^{-\alpha}} \theta^{-\frac{N}\beta(1-1/\eta)}\phi(\theta)\,d\theta\le c\left\{
\begin{array}{ll}
t^{-\frac{N\alpha}\beta(1-1/\eta)}&\text{ if } t\le1, \\ [2mm]
t^{-\alpha}&\text{ if } t\ge1,
\end{array}
\right.
$$
$$
I_2\le  ct^{-\frac{N\alpha}\varpi(1-1/\eta)}\int_{t^{-\alpha}}^\infty \theta^{-\frac{N}\varpi(1-1/\eta)}\phi(\theta)\,d\theta\le c\left\{
\begin{array}{ll}
t^{-\gamma}e^{-t^{-\gamma'}}&\text{ if } t\le1, \\ [2mm]
t^{-\frac{N\alpha}\varpi(1-1/\eta)}&\text{ if } t\ge1,
\end{array}
\right.
$$
for some $\gamma,\gamma'>0$. We have also needed $\frac N\beta(1-1/\eta)<1$ for the convergence of the first integral, which gives the restriction on $\eta$.
\end{proof}

As for $Y_t$ we only have to take care of the extra coefficient $t^{\alpha-1}\theta$ in the integral: here the condition for convergence is $\frac N\beta(1-1/\eta)-1<1$.

\begin{lema}\label{lem-Y-estim0}
Let $\mathcal{D}_t=\partial_t^\alpha$. Assuming hypotheses \eqref{J2} and \eqref{J3} it holds $Y_t\in L^\eta(\mathbb{R}^N)$ for $t>0$, where $1\le \eta<\frac N{(N-2\beta)_+}$ if $N\ge2\beta$, or $1\le \eta \le\infty$ if $N<2\beta$. Moreover
\begin{equation}\label{decayYt2}
\|Y_t\|_\eta\le c\begin{cases}
  t^{\alpha-1-\frac{\alpha N}{\beta}(1-1/\eta)} &t\le1, \\ t^{\alpha-1-\alpha\min\{2,\frac{ N}{\varpi}(1-1/\eta)\}}&t>1.
\end{cases}
\end{equation}
\end{lema}

\begin{cor}\label{cor-ZtYt2}
In the above hypotheses, the conclusion of Corollaries~\ref{cor-Zt} and~\ref{cor-Yt} hold with $\delta$ in the estimate for $Y_t*g$ replaced by
\begin{equation}\label{exp:u0yt}
\delta'=\delta'(r,q;t)=\begin{cases}
  \frac{ N}\beta(1/r-1/q) &t\le1, \\ \min\{2,\frac{N}\varpi(1/r-1/q)\}&t>1.
\end{cases}
\end{equation}
\end{cor}
As we will see later, for the values at which we evaluate these exponents we have $\delta'=\delta$.

\section{Basic theory for the nonlinear problem}\label{sec-basictheory}

\setcounter{equation}{0}

We start by considering the linear problem \eqref{prob-linear2}.

\subsection{The solution to the linear problem}\label{sec-linearproblem}

\begin{teo}\label{teo-exist}
For each data $u_0\in X$ and $f\in L^\infty((0,T);X)$ there exists a unique mild solution to problem \eqref{prob-linear2}. Moreover a comparison principle holds: if $u_0\le v_0$ and $f\le g$ are two pairs of admissible data, the corresponding solutions $u$ and $v$ satisfy $u\le v$ a.e. in~$Q_T$.
\end{teo}

\begin{proof} The proof is standard starting with the estimate of \eqref{mild0}
$$
\|u(\cdot,t)\|_X\le \|u_0\|_X+h_\ell(t)\max_{0\le \tau\le t}\|f(\cdot,\tau)\|_X
$$
which follows from  \eqref{eq:u0zyt-inf}. We omit the details.
\end{proof}

We also need the concept of very weak solution.

\begin{defi}
Given $u_0\in X$ and $f\in L^\infty((0,T);X)$, a function $u\in C([0,T);X)$ is said to be a very weak solution of problem \eqref{prob-linear2} if
\begin{equation}\label{weak}
\int_{Q_T}(u-u_0)({}_t\mathcal {D}_T\zeta)+\int_{Q_T}u\mathcal{L}_x\zeta=\int_{Q_T}f\zeta,
\end{equation}
for every test function $\zeta\in L^\infty([0,T);X)\cap C_{x,t}^{2,1}(Q_T)$ such that $\zeta(\cdot,T)\equiv0$, where ${}_t\mathcal{D}_T$ is the associated backwards Riemann-Liouville derivative,
$${}_t\mathcal{D}_T \phi(t)=-\partial_t\left(\int_t^T\phi(\tau)\kappa(\tau-t)\,d\tau\right),\qquad 0<t<T.
$$
\end{defi}
This definition is interpreted as: if $u$ is a strong solution, i.e., the equation in Problem~\eqref{prob-linear2} is satisfied pointwise, we multiply by a test function, pass the spatial operator to the test, and integrate by parts the integral involving the time operator, see for instance \cite{SamkoKilbasMarichev93}. But for $u$ being a strong solution much regularity is needed. We show that $u$ is a very weak solution if it is mild.

\begin{prop}\label{prop-veryweak}
  If $u$ is a mild solution to problem \eqref{prob-linear2}  then $u$ is a very weak solution.
\end{prop}
\begin{proof}
We write
$$
u(x,t)=Z_t*u_0(x)+\int_0^tY_{t-s}*f(x,s)\,ds
$$
as
$$
u-u_0=\left(Z_t*u_0-u_0\right)+\int_0^tY_{t-s}*f(s)\,ds=U_1+U_2.
$$
We perform a convolution in time of this identity with $\kappa$, multiply by a test function $\zeta$ as in the definition of very weak solution, differentiate with respect to $t$ and integrate in $Q_T$.

First the term on the left vanishes,
$$
\int_0^T\partial_t\big([\kappa\star(u-u_0)]\zeta\big)\,dt=[\kappa\star(u-u_0)]\zeta\Big|_0^T =0.
$$
Now
$$
\begin{array}{rl}
\displaystyle\int_{Q_T}\partial_t\left([\kappa\star (U_1+U_2)]\zeta\right)\,dtdx&=\displaystyle\int_{Q_T}\partial_t\left[\kappa\star (U_1+U_2)\right]\zeta\,dtdx \\ [3mm]
&\displaystyle+
\int_{Q_T}\left[\kappa\star (U_1+U_2)\right]\partial_t\zeta\,dtdx.
\end{array}
$$
For the first term on the right, by Plancherel identity if $\zeta=\widehat\phi\in L^2(\mathbb{R}^N)$, and  using \eqref{Z-Fourier} for $U_1$ and \eqref{Y-Fourier} for $U_2$,
$$
\begin{array}{rl}
\displaystyle\int_{\mathbb{R}^N}\partial_t\left[\kappa\star (U_1+U_2)\right]\zeta\,dx&=\displaystyle\int_{\mathbb{R}^N}\Big[\partial_t\left[\kappa\star (U_1+U_2)\right]\Big]^{\widehat{\;}}\phi\,dx\\ [3mm]
&=\displaystyle\int_{\mathbb{R}^N}(\widehat{f}-m\widehat{u})\phi\,d\xi\\ [3mm]
&=\displaystyle
\displaystyle\int_{\mathbb{R}^N}(\widehat{f}\phi-\widehat{u}m\phi)\,d\xi\\ [3mm]
&=\displaystyle
\displaystyle\int_{\mathbb{R}^N}(f\zeta-u\mathcal{L}_x\zeta)\,dx.
\end{array}$$
As to the second term (we omit the variable $x$),
$$
\begin{array}{rl}
\displaystyle\int_0^T\left[\kappa\star (U_1+U_2)\right]\partial_t\zeta\,dt&=\displaystyle
\int_0^T\int_0^t\kappa(t-s)(u(s)-u_0)\,ds\,\partial_t\zeta(t)\,dt\\ [3mm]
&=\displaystyle \int_0^T\int_s^T\kappa(t-s)\partial_t\zeta(t)\,dt\,(u(s)-u_0)\,ds\\ [3mm]
&=\displaystyle -\int_0^T{}_t\mathcal{D}_T\zeta(s)(u(s)-u_0)\,ds.
\end{array}$$
In summary,
$$
0=\int_{Q_T}(f\zeta-u\mathcal{L}_x\zeta)-\int_{Q_T}{}_t\mathcal{D}_T\zeta(u-u_0).
$$
\end{proof}

We also observe that if $\phi(t)=w(T-t)$ where $w\in C^1(0,T)$ satisfies $w(0)=0$, then the backwards derivative can be written as
\begin{equation}\label{eq:DtT}
  {}_t\mathcal{D}_T\phi(t)=w'\star\kappa(T-t).
\end{equation}

We now pass to study our nonlinear problem \eqref{eq.principal}.

\subsection{The nonlinear problem}\label{sec-nonlinearproblem}

\begin{defi}
  Given $u_0\in X$, a function $u\in C([0,T);X)$ is said to be a mild solution of problem \eqref{eq.principal} if
\begin{equation}\label{mild}
u(x,t)=Z_t*u_0(x)+\int_0^tY_{t-\tau}*u^p(x,\tau)\,d\tau,
\end{equation}
a.e $x\in\mathbb{R}^N$ and for all $0<t<T$.\end{defi}

By the previous section we know that mild solutions are very weak solutions in the sense that
\begin{equation}\label{weak}
\int_{Q_T}(u-u_0)({}_t\mathcal {D}_T\zeta)+\int_{Q_T}u\mathcal{L}_x\zeta=\int_{Q_T}u^p\zeta,
\end{equation}
for every test function $\zeta\in L^\infty([0,T);X)\cap C_{x,t}^{2,1}(Q_T)$ such that $\zeta(\cdot,T)\equiv0$

Existence of a unique mild solution for some $0<T\le\infty$ is proved by a fixed point argument.
\begin{teo}\label{teo-exist}
For each $u_0\in X$ there exists a unique mild solution  for some $T_0>0$, to problem \eqref{eq.principal}. Moreover a comparison principle holds: if $u_0\le v_0$ are two admissible data, the corresponding solutions $u$ and $v$ satisfy $u\le v$ a.e. in $Q_T$.
\end{teo}

\begin{proof} Let $\|u_0\|_X=\|u_0\|_1+\|u_0\|_\infty=M$ and, for $T_0>0$ fixed, consider the space
$$
E=\{v\in C((0,T_0);X)\,:\sup_{0<t<T_0}\|v(t)\|_X\le 4M\}.
$$
where abusing notation we omit the spatial variable $x$. In $E$ we use the standard distance $d(u,v)=\sup\limits_{0<t<T_0}\|u(t)-v(t)\|_X$.
We define the operator
\begin{equation}\label{operator-phi}
\Phi(v)(t)=Z_t*u_0+\int_0^tY_{t-\tau}*v^p(\tau)\,d\tau.
\end{equation}
We want to prove that if $T_0$ is small then $\Phi:E\to E$ is contractive, and thus has a unique fixed point. Using \eqref{eq:u0zyt-inf}, if $v\in E$ and $t<T_0$,
$$
\|\Phi(v)(t)\|_\infty\le \|u_0\|_\infty+\int_0^t\ell(t-\tau)\|v(\tau)\|_\infty^p\,d\tau\le M+cM^ph_\ell(T_0)\le2M,
$$
and also
$$
\|\Phi(v)(t)\|_1\le \|u_0\|_1+\int_0^t\ell(t-\tau)\|v(\tau)\|_\infty^{p-1}\|v(\tau)\|_1\,d\tau\le M+cM^ph_\ell(T_0)\le2M,
$$
provided $T_0$ is small enough.
Thus $\Phi(E)\subset E$. Similarly, for $v_1,v_2\in E$, and $t<T_0$,
$$
\begin{array}{rl}
\displaystyle\|\Phi(v_1)(t)-\Phi(v_2)(t)\|_X&\displaystyle\le \int_0^t\ell(t-\tau)(4M)^{p-1}\|v_1(\tau)-v_2(\tau)\|_{X}\,d\tau \\ [4mm]
&\displaystyle\le ch_\ell(T_0) M^{p-1}\sup\limits_{0<\tau<T_0}\|v_1(\tau)-v_2(\tau)\|_X\\ [3mm]
&\displaystyle\le \frac12d(v_1,v_2),
\end{array}$$
and $\Phi$ is contractive in $E$.

Let us prove the comparison, from where we deduce also uniqueness. Subtracting the definitions for two solutions $u$ and $v$ in  $Q_T$, for some $T>0$, we get
$$
u(t)-v(t)=Z_t*(u_0-v_0)+\int_0^t Y_{t-\tau}*\left(u^p(\tau)-v^p(\tau)\right)\,d\tau.
$$
Since both kernels, $Z_t$ and $Y_t$ are nonnegative, we obtain, if $u_0\le v_0$,
$$
\begin{array}{rl}
(u(t)-v(t))_+&\displaystyle=\int_0^t Y_{t-\tau}*\left(u^p(\tau)-v^p(\tau)\right)_+\,d\tau\\ [3mm]
&\displaystyle\le c\int_0^t Y_{t-\tau}*(\|u(\tau)\|^{p-1}_X+\|v(\tau)\|^{p-1}_X)\left(u(\tau)-v(\tau)\right)_+\,d\tau.
\end{array}
$$
We finish with Gronwall's inequality, $(u(t)-v(t))_+=0$ a.e.
\end{proof}

We have constructed a mild solution in $Q_{T_0}$, for some $T_0$ small, which is unique. Prolonging $u$ up to a maximal time $T$ of existence is  easy using uniqueness, i.e., if $T<\infty$ then $\lim\limits_{t\to T^-}\|u(t)\|_\infty=\infty$. Also a comparison principle between supersolutions and subsolutions is immediate to obtain from the previous arguments.

\begin{rem} $i)$ It is clear that a comparison principle also holds for subsolutions and supersolutions, if these are understood by putting inequalities in the difinition formula~\eqref{mild}.

$ii)$ With the only assumption \eqref{J0} we do not have any smoothing effect, that is, even the first term $Z_t*u_0$ is not better than $u_0$. Thus to get better regularity to the solution~\eqref{mild} we need to impose it to the initial value. On the other hand, a smoothing effect will be crucial in the construction of global solutions in Section~\ref{sec-bup2}, so there we must assume condition~\eqref{J3}.
\end{rem}

\section{Blow-up}\label{sec-bup1}

\setcounter{equation}{0}

We prove in this section the first part of Theorem \ref{teo-BUP}, the existence of blow-up for any $p>1$ and the nonexistence of global solutions for small $p$. We will use the very weak formulation~\eqref{weak}, so we first need to estimate the action of the diffusion operator over some special test functions. See \cite[Lemma 2.1]{BonforteVazquez14} for the fractional Laplacian. The first result is very general and has interest by itself. Later on we obtain a more precise result.

\begin{prop}\label{prop-BV}
Let $\varphi\in C^2(\mathbb{R}^N)\cap L^1(\mathbb{R}^N)$ be a positive, radially symmetric and nonincreasing function such that $|D^2\varphi(x)|\le c|x|^{-2}\varphi(x)$. Then $\mathcal{L}_x\varphi\in L^\infty(\mathbb{R}^N)$ and for $|x|$ large it holds
$$
|\mathcal{L}_x\varphi(x)|\le c(\varphi(x/2)+\mathcal{J}(x/2)).
$$
\end{prop}

\begin{proof}
The fact that $\mathcal{L}_x\varphi$ is finite is easy from the fact that the kernel is of L\'evy type, so we concentrate in estimating its value for large $|x|$. As in \cite{BonforteVazquez14} we divide the integral defining $\mathcal{L}_x\varphi$ into four integrals
$$
\mathcal{L}_x\varphi(x)=\int_{\mathbb{R}^N}(\varphi(x)-\varphi(z))\mathcal{J}(x-z)\,dz=\sum_{i=1}^4I_i
$$
where
$$
I_i=\int_{\Omega_i}(\varphi(x)-\varphi(z))\mathcal{J}(x-z)\,dz,
$$
and
$$
\begin{array}{ll}
\Omega_1=\{|x-z|<|x|/2\},\quad
&\Omega_2=\{|z|<|x|/2\},\\ [3mm]
\Omega_3=\{|z|>3|x|/2\},\quad
&\Omega_4=\{|x|/2<|z|<3|x|/2\}-\Omega_1.
\end{array}
$$
We have
$$
\begin{array}{l}
|I_1|\displaystyle=\left|\int_{|x-z|<|x|/2}\nabla\varphi(x)(x-z)\mathcal{J}(x-z)\,dz\right.\\ [4mm]
\hspace{7mm}\displaystyle+
\left.\frac12\int_{|x-z|<|x|/2}\langle D^2\varphi(\overline x)(x-z),x-z\rangle\mathcal{J}(x-z)\,dz\right| \\ [4mm]
\hspace{7mm}\displaystyle\le0+\sup_{|\overline x|>|x|/2}|D^2\varphi(\overline x)|\int_{|w|<|x|/2}|w|^2\mathcal{J}(w)\,dw\\ [4mm]
\hspace{7mm}\displaystyle\le c|x|^{-2}\varphi(x/2)\left(\int_{|w|<1}|w|^2\mathcal{J}(w)\,dw+|x|^2\int_{1<|w|<|x|/2}\mathcal{J}(w)\,dw\right)\\ [4mm]
\hspace{7mm}\displaystyle\le c|x|^{-2}\varphi(x/2)(1+|x|^2)\le c\varphi(x/2).
\end{array}
$$
The integral with the gradient is zero by symmetry; the last two integrals are bounded by ~\eqref{J0}.
$$
\begin{array}{l}
|I_2|\displaystyle
= \int_{|z|<|x|/2}(\varphi(z)-\varphi(x))\mathcal{J}(x-z)\,dz\le c\mathcal{J}(x/2)\int_{|z|<|x|/2}\varphi(z)\,dz\\ [4mm]
\hspace{7mm}\displaystyle
\le c\mathcal{J}(x/2).
\end{array}
$$
We have used $|\varphi(x)-\varphi(z)|=\varphi(z)-\varphi(x)\le \varphi(z)$, and $|x-z|\ge|x|/2$.
$$
|I_3|
= \int_{|z|>3|x|/2}(\varphi(x)-\varphi(z))\mathcal{J}(x-z)\,dz\le \varphi(x)\int_{|z|>3|x|/2}\mathcal{J}(z/3)\,dz.
$$
Here $|\varphi(x)-\varphi(z)|=\varphi(x)-\varphi(z)\le \varphi(x)$, and $|x-z|\ge|z|/3$.
$$
|I_4|\le 2\varphi(x/2)\int_{|w|>|x|/2}\mathcal{J}(w)\,dw.
$$
And here $|\varphi(x)-\varphi(z)|\le\varphi(x)+\varphi(z)\le 2\varphi(x/2)$. We obtain the desired estimate.
\end{proof}

We now take a more precise function $\phi$ and focus on kernels satisfying \eqref{J1}.
\begin{prop}\label{prop2-BV}
Let  $\mathcal{J}$ satisfy \eqref{J1} with $\gamma<2$, and let $\phi$ be defined by $\phi(x)=\psi(x/A)$ for some $A>2$, where $\psi(z)=1$ for $|z|\le1$, $\psi(z)=|z|^{-N-\gamma}$ for $|z|\ge2$, is such that $\phi$ satisfies the hypothesis of Proposition~\ref{prop-BV}. Then it holds
\begin{equation}\label{phi-A}
|\mathcal{L}_x\phi(x)|\le cA^{-\gamma}\phi(x)
\end{equation}
for every $x\in\mathbb{R}^N$.
\end{prop}

\begin{proof}
First, if $|x|<A/2$ we have,
$$
\begin{array}{rl}
0<\mathcal{L}_x\phi(x)&\displaystyle=\int_{|z|>A}(1-\phi(z))\mathcal{J}(x-z)\,dz \\ [3mm]
&\displaystyle\le
c\int_{|z|>A}|z|^{-N-\gamma}\,dz=cA^{-\gamma}=cA^{-\gamma}\phi(x),
\end{array}
$$
since $|x-z|>|z|-|x|>A/2>1$. Now let $|x|>A/2$, and estimate the integrals in the notation of Proposition~\ref{prop-BV}.
$$
\begin{array}{rl}
|I_1|&\displaystyle\le \sup_{|\overline x|>|x|/2}|D^2\varphi(\overline x)|\int_{|w|<|x|/2}|w|^2\mathcal{J}(w)\,dw \\ [3mm]
&\displaystyle\le cA^{-2}(|x|/A)^{-N-\gamma-2}\left(\int_{|w|<1}|w|^2\mathcal{J}(w)\,dw+\int_{1<|w|<|x|/2}|w|^2\mathcal{J}(w)\,dw\right)\\ [4mm]
&\displaystyle\le cA^{-2}(|x|/A)^{-N-\gamma-2}(1+|x|^{2-\gamma})\le cA^{-\gamma}(|x|/A)^{-\gamma}(|x|/A)^{-N-\gamma}\\ [4mm]
&\displaystyle\le cA^{-\gamma}\phi(x).
\end{array}
$$
On the other hand,
$$
\begin{array}{rl}
|I_2|&\displaystyle\le  c\mathcal{J}(x/2)\int_{|z|<|x|/2}\phi(z)\,dz\le cA^N|x|^{-N-\gamma}\le cA^{-\gamma}\phi(x),\\ [3mm]
|I_3|&\displaystyle\le \phi(x)\int_{|z|>3|x|/2}\mathcal{J}(z/3)\,dz\le c(|x|/A)^{-N-\gamma}|x|^{-N-\gamma}\le cA^{-\gamma}\phi(x),\\ [3mm]
|I_4|&\displaystyle\le 2\phi(x/2)\int_{|w|>|x|/2}\mathcal{J}(w)\,dw\le cA^{-\gamma}\phi(x).
\end{array}
$$

\end{proof}

\begin{teo}\label{teo-fujita2}
Assume hypothesis \eqref{J1}. If $p>1$ then the solution to problem \eqref{eq.principal} blows up in finite time if the datum $u_0$ is large enough
\end{teo}

\begin{proof}
The proof follows the standard argument of \cite{Kaplan63} of using as a test in the very weak formulation a function which plays the role of an eigenfunction of the operator. We use also ideas of \cite{ZhangSun15} for the case of a Caputo time derivative and local diffusion.

We first observe that if $\mathcal{J}$ satisfies condition \eqref{J1} with some $\gamma\ge2$ then it also satisfies the same condition with $\gamma=2-\epsilon$ for any $0<\epsilon<2$; we thus assume without loss of generality $\gamma<2$.

Suppose the solution $u$ is defined at least up to $t=1$. Define the  test function
$$
\zeta(x,t)=\varphi_1(x)\varphi_2(t),
$$
where $\varphi_1=\phi$, the function in Proposition \ref{prop2-BV} with $A=3$, and
$$
\varphi_2(t)=h_\ell^\mu(1-t), \qquad\mu=\frac p{p-1}.
$$

Let us estimate the action of the operator $\mathcal{H}$ on $\zeta$, i.e., we estimate $\mathcal{L}_x\varphi_1$ and ${}_t\mathcal{D}_T\varphi_2$.

By \eqref{phi-A} we have $|\mathcal{L}_x\varphi_1(x)|\le c\varphi_1(x)$.
On the other hand, using \eqref{eq:DtT}, since $h_\ell$ is nondecreasing, we obtain
$$
\begin{array}{rl}
{}_t\mathcal{D}_1\varphi_2(t)&= \Big((h_\ell^{\mu})'\star\kappa\Big)(1-t)=\mu \Big((h_\ell^{\mu-1}\ell)\star\kappa\Big)(1-t)\\ [2mm]
&\le\mu h_\ell^{\mu-1}(1-t)(\ell\star\kappa)(1-t)\\ [2mm]
&=\mu h_\ell^{\mu-1}(1-t)
=c\varphi_2^{1/p}(t).
\end{array}
$$
Also,
$$
\begin{array}{rl}
\displaystyle\int_0^1{}_t\mathcal{D}_1\varphi_2(t)\,dt&\displaystyle=\int_0^1\int_0^{1-t}(h_\ell^\mu)'(\tau)\kappa(1-t-\tau)\,d\tau dt \\ [3mm]
&\displaystyle=\int_0^1\kappa(z)\int_0^{1-z}(h_\ell^\mu)'(1-t-z)\,dtdz \\ [3mm]
&\displaystyle=\int_0^1\kappa(z)h_\ell^\mu(1-z)\,dz\ge c>0.
\end{array}
$$
Thus the identity
$$
\int_{Q_1}u_0({}_t\mathcal {D}_1\zeta)=\int_{Q_1}u({}_t\mathcal {D}_1+\mathcal{L}_x)\zeta-
\int_{Q_1}u^p\zeta$$
i.e.,
$$
\int_{Q_1}u_0\varphi_1({}_t\mathcal {D}_1\varphi_2)=\int_{Q_1}(u\varphi_1({}_t\mathcal {D}_1\varphi_2)+u\varphi_2\mathcal{L}_x\varphi_1)-
\int_{Q_1}u^p\zeta$$
becomes
$$
\begin{array}{rl}
\displaystyle c\int_{\mathbb{R}^N}u_0\varphi_1&\displaystyle \le  c\int_{Q_1}u\varphi_1\varphi_2^{1/p}+c\int_{Q_1}u\varphi_1\varphi_2-
\int_{Q_1}u^p\varphi_1\varphi_2 \\ [3mm]
&\displaystyle \le  c\int_{Q_1}u\varphi_1\varphi_2^{1/p}-
\int_{Q_1}u^p\varphi_1\varphi_2.
\end{array}
$$
H\"older's inequality on the first integral of the right-hand side gives
$$
\int_{Q_1}u\varphi_1\varphi_2^{1/p}
\le \left(\int_{Q_1}u^p\varphi_1\varphi_2\right)^{\frac1p}
\left(\int_{Q_1}\varphi_1\right)^{\frac{p-1}p}=c\left(\int_{Q_1}u^p\varphi_1\varphi_2\right)^{\frac1p}.
$$
Thus we get
$$
\int_{|x|<3}u_0\le\int_{\mathbb{R}^N}u_0\varphi_1\le c\left(\int_{Q_1}u^p\zeta\right)^{1/p}-
\int_{Q_1}u^p\zeta\le c.
$$
This is a contradiction if $\int_{|x|<3}u_0$ is large, and then $u$ must blow up before $t=1$.
\end{proof}

\begin{rem}
  Without imposing condition \eqref{J1} to the kernel $\mathcal{J}$, that is, considering L\'evy kernels in the limit of integrability at infinity, for instance
  $$
  \mathcal{J}(z)=|z|^{-N}(\log|z|)^{-\varepsilon},\qquad|z|>2,\quad\varepsilon>1,
  $$
  we could use the argument of Proposition \eqref{prop-BV} if we assume additional regularity properties in order to use the kernel itself as test function $\varphi_1$.
\end{rem}

We now characterize the Fujita range, i.e., the interval of powers $p$ such that all solutions blow up.

\begin{teo}\label{teo-fujita2}
Assume hypotheses \eqref{K1} and \eqref{J1}. If $1<p<1+\overline\gamma/N$ then for any given data $u_0$ the solution to problem \eqref{eq.principal} blows up in finite time.
\end{teo}

\begin{proof}
We repeat the previous argument, using this time a rescaled test function, following the original idea of \cite{Fujita66}, and this is why we require condition \eqref{K1}. Suppose $u$ is a global solution, so it satisfies identity~\eqref{weak} for any $T>0$.

Consider first the case $\gamma<2$. For $T>0$ fixed define now the  test function
$$
\zeta(x,t)=\varphi_1(x)\varphi_2(t),
$$
where $\varphi_1=\phi$, the function in Proposition \ref{prop2-BV} with $A=T^{\alpha/\gamma}$, and
$$
\varphi_2(t)=(1-t/T)_+^\mu, \qquad\mu=\frac{p\alpha}{p-1}.
$$
These functions satisfy $0<\varphi_1(x)\le1$, $\varphi_1(x)=1$ for $|x|\le T^{\alpha/\gamma}$; $0\le\varphi_2(t)\le1$, $\varphi_2(T)=0$. We estimate as before $\mathcal{H}\zeta$ for $T$ large.

From Proposition \ref{prop2-BV} we have
$$
\mathcal{L}_x\varphi_1(x)\le cT^{-\alpha}\varphi_1(x).
$$
On the other hand, using \eqref{K1},
$$
\begin{array}{l}
{}_t\mathcal{D}_T\varphi_2(t)\le c T^{-\alpha}(1-t/T)^{\mu-\alpha}_+= c T^{-\alpha}\varphi_2^{1/p}(t), \\ [3mm]
\displaystyle\int_0^T{}_t\mathcal{D}_T\varphi_2(t)\ge c\int_0^TT^{-\alpha}(1-t/T)^{\mu-\alpha}_+=cT^{1-\alpha}.
\end{array}
$$
Thus the very weak formulation becomes here
$$
cT^{1-\alpha}\int_{\mathbb{R}^N}u_0\varphi_1\le  cT^{-\alpha}\int_{Q_T}u\varphi_1\varphi_2^{1/p}-
\int_{Q_T}u^p\varphi_1\varphi_2.
$$
H\"older's inequality on the first integral of the right-hand side gives
$$
\begin{array}{rl}
\displaystyle\int_{Q_T}u\varphi_1\varphi_2^{1/p}&\displaystyle\le \left(\int_{Q_T}u^p\varphi_1\varphi_2\right)^{\frac1p}
\left(\int_{Q_T}\varphi_1\right)^{\frac{p-1}p}\\ [3mm]
&\displaystyle\le cT^{(1+\frac{\alpha N}\gamma)\frac{p-1}p}\left(\int_{Q_T}u^p\varphi_1\varphi_2\right)^{\frac1p}.
\end{array}$$
Thus we get
$$
cT^{1-\alpha}\int_{\mathbb{R}^N}u_0\varphi_1\le cT^{-\alpha+(1+\frac{\alpha N}\gamma)\frac{p-1}p}\left(\int_{Q_T}u^p\zeta\right)^{1/p}-
\int_{Q_T}u^p\zeta.
$$
Maximizing the right-hand side we obtain
\begin{equation}\label{estim-fujita}
\int_{\mathbb{R}^N}u_0\varphi_1\le cT^{\alpha-1}\left(T^{-\alpha+(1+\frac{\alpha N}\gamma)\frac{p-1}p}\right)^{\frac p{p-1}}=cT^{\alpha(\frac N\gamma-\frac1{p-1})}.
\end{equation}
Since the exponent is negative precisely for $p<1+\gamma/N$, taking the limit $T\to\infty$ if $u$ is global in time we conclude
$$
\int_{\mathbb{R}^N}u_0=0,
$$
which is a contradiction. If $\gamma=2$ we replace $\gamma$ by $\gamma-\epsilon$ and we get the same result.

In the case $\gamma>2$ the previous argument works using the spatial test function $\varphi_1(x)=e^{-T^{-\alpha}|x|^2}$: it is easy to check that $\mathcal{L}_x\varphi_1(x)\le cT^{-\alpha}\varphi_1(x)$. Actually
$$
\begin{array}{rl}
\mathcal{L}_x (e^{-\lambda|x|^2})&\displaystyle=\int_{\mathbb{R}^N}\left(e^{-\lambda|x|^2}-e^{-\lambda|x-y|^2}\right)\mathcal{J}(y)\,dy \\ [3mm]
&\displaystyle=e^{-\lambda|x|^2}\int_{\mathbb{R}^N}\left(1-e^{-\lambda(|y|^2-2xy)}\right)\mathcal{J}(y)\,dy\\ [3mm]
&\displaystyle\le e^{-\lambda|x|^2}\int_{\mathbb{R}^N}\lambda(|y|^2-2xy)\mathcal{J}(y)\,dy\\ [3mm]
&\displaystyle
\le c\lambda e^{-\lambda|x|^2},
\end{array}
$$
since the first integral is bounded by hypothesis and the second one vanishes by symmetry (use Principal Value if $\mathcal{J}$ is singular at the origin). We then get inequality \eqref{estim-fujita} with $\gamma$ replaced by $2$. The proof is done.
\end{proof}

As a corollary of the proof we obtain an estimate of the blow-up time in terms of an integral condition on the initial datum $u_0$. See \cite{CortazarQuirosWolanski24} for the double fractional equation~\eqref{double-eq}.

\begin{teo}\label{teo-kaplan}
There is a constant $c>0$ such that if
\begin{equation}\label{cqw}
\int_{|x|<T_0^{\alpha/\overline\gamma}}u_0>cT_0^{\alpha(\frac N{\overline\gamma}-\frac1{p-1})}
\end{equation}
then the solution blows up before $t=T_0$.
\end{teo}

\begin{proof}
Just observe that, with $\varphi_1$ defined in the previous proof,
$$
\int_{\mathbb{R}^N}u_0\varphi_1\ge c\int_{|x|<T_0^{\alpha/\overline\gamma}}u_0,
$$
so the above condition results in a contradiction with inequality \eqref{estim-fujita} if $u$ is defined in $[0,T_0]$.
\end{proof}

\section{Global solutions for large $p$}\label{sec-bup2}

\setcounter{equation}{0}

We now look for constructing global solutions for large $p$ and small initial value. To guess the critical exponent $p$ we take into account the two typical arguments used in the literature for homogeneous problems. This should work for equation~\eqref{double-eq}, but not directly for our non homogeneous general equation~\eqref{eq.principal}. Nevertheless it could serve as a conjecture.

The first argument is to compare  the diffusion exponent decay with the ODE blow-up exponent, see \cite{DengLevine00}. For instance in the semilinear heat equation \eqref{local-eq}, the diffusion decay rate is $\rho_d=N/2$ (the decay of the Gauss kernel), while the solutions of the corresponding ODE $\partial_tU=U^p$ are $U(t)=c(T-t)^{-\frac1{p-1}}$, so $\rho_r=1/(p-1)$. Thus $\rho_d=\rho_r$ gives $p_*=1+2/N$.
The same can be applied for the double fractional equation~\eqref{double-eq}, but only if $\alpha=1$ or $\beta>N$, since otherwise the fundamental solution is not bounded at $x=0$, see \cite{KemppainenSiljanderVergaraZacher16}. When $0<\alpha<1=N<\beta<2$ the fundamental solution decay is $\|Z_t\|_\infty\le ct^{-\frac\alpha{\beta}}$, see~\eqref{selfsimilarZ}, while the solution to $\partial_t^\alpha U=U^p$ satisfies $U(t)\sim(T-t)^{-\frac\alpha{p-1}}$; this produces $p_*=1+\beta$.

The second idea, introduced in \cite{CazenaveDicksteinWeissler08}, is to use the invariance of some $L^q$ norm under the natural scaling of the equation. If $u$ is a solution to the equation~\eqref{double-eq}, so is $u_\lambda(x,t)=\lambda^{\frac1{p-1}}u(\lambda^{\frac1\beta} x,\lambda^{\frac1\alpha} t)$ for every $\lambda>0$, and $\|u_\lambda\|_{q_0}=\|u\|_{q_0}$ only for $q_0=N(p-1)/\beta$. In the semilinear case $\alpha=1$, $\beta=2$, it was shown in \cite{CazenaveDicksteinWeissler08} that if $\|u_0\|_{q_0}$ is small then the solution is global in time, and the condition $q_0>1$ is $p>1+2/N$, the Fujita exponent. In the double fractional case~\eqref{double-eq} this argument gives  $p_*=1+\beta/N$.

Though our equation is not invariant under any rescaling, if the operator behaves like that of equation~\eqref{double-eq}, the rescaled function $u_\lambda$ is a solution of a different equation but with a differential operator satisfying the same properties. The critical exponent should then be the same.  The major problem in the general situation of equation \eqref{eq.principal} is the different behaviour of the kernel $\mathcal{J}$ at the origin and at infinity, which is reflected in the different decay rates of $Z_t$ for $t$ small or $t$ large, see~\eqref{decayZt}. We now show that only the behaviour at infinity of $\mathcal{J}$ matters, that which gives the behaviour for large $t$,  but with a minimum of singularity at the origin: above the exponent $p_*=1+\varpi/ N$ there exist small global solutions.

What we need in the construction of global solutions is an estimate of the kernels $Z_t$ and $Y_t$ so as to have a smoothing effect for the solutions to our problem if the initial datum is small in some norm. We have proved in Section~\ref{sec-fundamentalair} that such estimates hold true under the assumption that $N<2\beta$ or $\mathcal{D}_t=\partial_t^\alpha$. We have then the following result that implies the last two items in Theorem~\ref{teo-BUP}.

\begin{teo}\label{teo-fujita3}
Assume hypotheses \eqref{K1}, \eqref{J2} and \eqref{J3} with $\beta\ge\frac{N\varpi}{N+\varpi}$. Assume also $N<2\beta$ if $\mathcal{D}_t\neq \partial_t^\alpha$. If $p\ge 1+ \varpi/ N$  then the solution to problem \eqref{eq.principal} is uniformly bounded for every time $0\le t<\infty$ provided the datum $\|u_0\|_1$ is small enough.
\end{teo}

\begin{proof} We follow the technique of Weissler \cite{Weissler81}, as applied in  \cite{ZhangSun15} for the Caputo equation with local diffusion, $\mathcal{L}=-\Delta$.
We first observe that it is enough to prove the results for $p=1+\varpi/ N$. In fact, if $v$ is a bounded global solution for some $p$, then $\varphi=M^{\frac{p-p'}{p'-1}}v$ is a (mild) supersolution to the equation with $p'>p$, where $M=\|v\|_\infty$. Thus the result is true in the interval $[p,\infty)$.

Observe that in the hypotheses considered, the estimates in Sections~\ref{sec-Fourier} and~\ref{subordination} hold.

Fix then $p=1+\varpi/ N$ and some $q>\max\{p,\varpi/\beta\}$. Define the function space
$$
\mathcal{A}=\{v \text{ measurable in $Q_\infty$ such that } \|v\|_{\mathcal{A}}<\infty\},
$$
where
$$
\|v\|_{\mathcal{A}}=\sup_{t>0}(1+t)^\theta\|v(t)\|_q,\qquad
  \theta=\frac{\alpha N(q-1)}{\varpi q}.
  $$

Define also the map $\Phi:\mathcal{A}\to\mathcal{A}$, as in the proof of Theorem \ref{teo-exist},
$$
\Phi(v)(t)=Z_t*u_0+\int_0^tY_{t-\tau}*v^p(\tau)\,d\tau.
$$
We want to prove that if $\|u_0\|_{1}=\varepsilon$ is small, and $R>0$ is taken also small, then $\Phi$ is contractive in the ball $B_R\subset\mathcal{A}$, and thus has a unique fixed point $u$. This is a mild solution to our problem, which is bounded locally in time by the existence result, Theorem~\ref{teo-exist}, with finite $L^q$-norm for all positive times. Then we improve the exponent $q$ to get $u$ in $L^r$ for some $r>N/p\beta$ which implies finally  that the solution  is bounded for all times by Corollaries~\ref{cor-Yt} or~\ref{cor-ZtYt2}. We only have to consider the estimates in the first case because, as we will see, the exponents $\delta$ and $\delta'$ coincide when putting $r=q/p$.

We develop this procedure in several steps.

\underline{\sc Step 1}. Let $\|u_0\|_\infty=M$.
Then, if $0<t<1$,
$$
\|Z_t*u_0\|_q\le\|u_0\|_q\le M^{\frac{q-1}q}\varepsilon^{\frac{1}q}.
$$
Now, for $t>1$, using estimate \eqref{eq:u0zt},
$$
\|Z_t*u_0\|_q\le  c t^{-\delta(1,q;t)}\varepsilon,
$$
provided $q<q_1=\frac{N}{(N-\beta)_+}$. Recall that for $t>1$ it is
$$
\delta(1,q;t)=\alpha\min\{1,\frac{N(q-1)}{\varpi q}\}=\frac{\alpha N(q-1)}{\varpi q}
$$
if we also impose $q\le q_2=\frac{N}{(N-\varpi)_+}$. We get in that way $\delta(1,q;t)=\theta$.  We therefore conclude $$
\|Z_t*u_0\|_{\mathcal{A}}\le c(M^{\frac{q-1}q}\varepsilon^{\frac1q}+\varepsilon).
$$

Now we have to estimate the second term in the $L^q$ norm of $\Phi(v)$ for $v$ in the ball $B_R$. We use $v^p(t)\in L^{q/p}$ with $\|v^p(t)\|_{q/p}\le R^p(1+t)^{-p\theta}$. Then, whenever $q<\frac{Nq/p}{N-\beta q/p}$ we can use estimate~\eqref{eq:u0yt}, and this is true since $q>\varpi/\beta$. We thus get
$$
I(t)=\left\|\int_0^tY_{t-\tau}*v^p(\tau)\,d\tau\right\|_q\le
c R^p\int_0^t(t-\tau)^{-a(t-\tau)}(1+\tau)^{-p\theta}\,d\tau,
$$
where $a(s)=1-\alpha+\delta(q/p,q;s)$, see \eqref{exp:u0zt}. We write $
a(s)=\begin{cases}
a_1&s<1,\\ a_2&s\ge1,
\end{cases}$ with the explicit values
$$
a_1=1-\alpha\left(1-\frac{\varpi}{\beta q}\right),\qquad a_2= 1-\alpha\left(1-\frac{1}q\right).
$$
Notice that here $\delta(q/p,q;s)=\delta'(q/p,q;s)$. Clearly $0<a_1,a_2<1$.

Now if $0<t<1$,
$$
I(t) \le c R^p\int_0^t  (t-\tau)^{-a_1}(1+\tau)^{-p\theta}\,d\tau\le cR^pt^{1-a_1}\le cR^p.
$$

Analogously $I(t)$ is bounded if $1\le t\le 2$. Let us then consider the case $t>2$. We have
$$
I(t)\le cR^p\left(\int_0^1+\int_1^{t-1}+\int_{t-1}^t\right)=cR^p\left(I_1+I_2+I_3\right).
$$
$$
\begin{array}{l}
\displaystyle I_1
\le\int_0^1(t-\tau)^{-a_2}\,d\tau\le c(t/2)^{-a_2};\\ [3mm]
\displaystyle I_2\le\int_1^{t-1}(t-\tau)^{-a_2}\tau^{-p\theta}\,d\tau=
t^{-a_2-p\theta+1}\int_{1/t}^{1-1/t}(1-s)^{-a_2}s^{-p\theta}\,ds\\ [3mm]\displaystyle\quad\le
ct^{-a_2-p\theta+1}=ct^{-\theta};\\ [3mm]
\displaystyle I_3\le\int_{t-1}^t(t-\tau)^{-a_1}\tau^{-p\theta}\,d\tau\le c(t/2)^{-p\theta}\le ct^{-\theta}.
\end{array}$$

To bound the second integral we use $a_2,p\theta<1$. We have seen that the first condition is true; as for the second condition this holds if $q< q_3=\frac{\alpha p}{\alpha p-p+1}$.

In order to obtain the desired estimate in the norm of $\mathcal{A}$, we need $a_2\ge\theta$, and this is verified since $\theta-a_2 =p\theta-1<0$. In summary, if we fix the exponent
\begin{equation}\label{room-q}
\max\{p,\varpi/\beta\}<q<\min\{q_1,q_2,q_3\},
\end{equation}
we have $I(t) \le cR^pt^{-\theta}$, which gives
$$
\|\Phi(v)\|_{\mathcal{A}}\le c(\varepsilon+\varepsilon^{\frac{1}q}+R^p)<R
$$
by taking $\varepsilon$ and $R$ small. This implies $\Phi(B_R)\subset B_R$. What is left is to check that there is room to choose the exponent $q$ in \eqref{room-q}.

\

First we have that $q_1>p$ is equivalent to $\frac{N}{N-\beta}>\frac{N+\varpi}{N}$, or $\beta>\frac{N\varpi}{N+\varpi}$; while $q_2>p$ and $q_3>p$ are trivial.

Notice that the previous restriction on $\beta$ implies
$\frac{\varpi}{\beta}\le 1+\frac{\varpi}{N}=p$. So, $q_i>{\varpi}/{\beta}$ is not a restriction.

\

\underline{\sc Step 2}. We now calculate, for $v_1,v_2\in B_R$,
$$
\begin{array}{l}
\displaystyle\|\Phi(v_1)(t)-\Phi(v_2)(t)\|_q\le\int_0^tc\left\|Y_{t-\tau}*(v_1^p(\tau)-v_2^p(\tau))\right\|_q\,d\tau \\ [2mm]
\displaystyle\quad\le\int_0^tc(t-\tau)^{-a(t-\tau)}\left(\|v_1(\tau)\|_q^{p-1}+\|v_2(\tau)\|_q^{p-1}\right)\|v_1(\tau)-v_2(\tau)\|_q\,d\tau \\ [2mm]
\displaystyle\quad\le\int_0^tc(t-\tau)^{-a(t-\tau)}2(R(1+\tau)^{-\theta})^{p-1}(1+\tau)^{-\theta}\|v_1-v_2\|_{\mathcal{A}}\,d\tau\\ [2mm]
\displaystyle\quad= cR^{p-1}\|v_1-v_2\|_{\mathcal{A}}\int_0^t(t-\tau)^{-a(t-\tau)}(1+\tau)^{-p\theta}\,d\tau\\ [2mm]
\displaystyle\quad\le  cR^{p-1}\|v_1-v_2\|_{\mathcal{A}}(1+t)^{-\theta},
\end{array}$$
exactly as before. Then $\Phi$ is contractive in $B_R$ if $R$ is small enough. Let $u$ be the unique fixed point of $\Phi$  in $B_R$.

\underline{\sc Step 3}. If $U\in C((0,T_0);X)$ is the solution constructed in Theorem~\ref{teo-exist}, by uniqueness we have $u=U$ for $0<t<T_0$, and thus
$$
u\in C((0,T_0);X)\cap L^\infty((0,\infty);L^q(\mathbb{R}^N).
$$
Assume for simplicity of the subsequent calculus that $T_0\ge2$. Let us now write the solution for $t>T_0$ through the expression
$$
\begin{array}{rl}
u(t)&\displaystyle=Z_t*u_0+\int_0^{1}Y_{t-\tau}*u^p(\tau)\,d\tau+\int_{1}^tY_{t-\tau}*u^p(\tau)\,d\tau \\ [3mm]
&\displaystyle=I_1(t)+I_2(t)+I_3(t).
\end{array}
$$
It is clear that $I_1$ is smooth and satisfies
$$
\|I_1(t)\|_\infty\le \|u_0\|_\infty,\qquad \|I_1(t)\|_q\le cM^{\frac{q-1}q}\varepsilon^{\frac{1}q}.
$$
As for the second integral,
$$
\begin{array}{l}
\displaystyle\|I_2(t)\|_\infty\le \sup_{0<\tau<1}\|u(\tau)\|_\infty^p\int_0^1c(t-\tau)^{\alpha-1}\,d\tau\le c\sup_{0<\tau<1}\|u(\tau)\|_\infty^p\\ [3mm]
\displaystyle\|I_2(t)\|_q\le\|u(t)\|_q\le R.
\end{array}
$$
Therefore
$$
I_1+I_2\in L^\infty((0,\infty);L^r(\mathbb{R}^N))
$$
for every $q\le r\le\infty$, and we can concentrate on estimating $I_3$.

\underline{\sc Step 4}. We now apply estimate \eqref{eq:u0yt} to the function $g(t)=u^p(t)\in L^{q/p}(\mathbb{R}^N)$. First, if $q/p>N/\beta$ we are done, since then
$$
\|I_3(t)\|_\infty\le \int_{1}^t(t-\tau)^{-1+\alpha-\delta(p/q,\infty;t-\tau)}R^p\tau^{-\theta p}\,d\tau\le c .
$$
On the contrary, if $q/p<N/\beta$, we have
$$
\|I_3(t)\|_r\le \int_{1}^t(t-\tau)^{-1+\alpha-\delta(p/q,r;t-\tau)}R^p\tau^{-\theta p}\,d\tau\le c,
$$
for every
$$
q\le r<\frac{Nq/p}{N-\beta q/p}=\frac{Nq}{Np-\beta q}.
$$
We get a smoothing effect $L^q\to L^r$. We now take $r_0=q$, $r_k=\chi r_{k-1}$ with $1<\chi <\frac{N}{Np-\beta q}$ fixed, and repeat the procedure, getting in each step the smoothing effect $L^{r_{k-1}}\to L^{{r_k}}$. Observe that at each step
$$
r_{k-1}< r_k<\frac{N r_{k-1}}{Np-\beta r_{k-1}}.
$$
We stop when we reach a value $k_*\ge1$ for which $\chi^{k_*} q>Np/\beta$, thus giving the smoothing effect $L^{r_{k^*}}\to L^\infty$. This implies  $u$ bounded.
\end{proof}

\

\appendix
\section{The nonlocal heat equation}\label{app-A}
\setcounter{equation}{0}
\renewcommand{\theequation}{A.\arabic{equation}}

In this appendix we review some easy estimates for the fundamental solution of the nonlocal heat equation without memory, which is used through a subordination formula, to study the problem with memory. That is, we characterize the function $G_t$ solution to
\begin{equation}\label{zeta}
\left\{
\begin{array}{ll}
(\partial_t +\mathcal{L}_x) G_t=0,\qquad & \text{in } Q,\\
G_0=\delta,
\end{array}\right.
\end{equation}
where $\delta$ is the Dirac mass at the origin. Since the Fourier transform is $\widehat G_t(\xi)=e^{-m(\xi)t}$,  we first need to study the symbol $m(\xi)$ of the diffusion operator $\mathcal{L}_x$. See for example \cite{KassmannMimica17,BrandledePablo18}.

\subsection{The symbol of $\mathcal{L}_x$}\label{sec-symbol}

Let $m(\xi)$ be the symbol given by \eqref{symbol}.
We first observe that it is given  by the L\'evy-Khintchine formula
\begin{equation}\label{kinchin}
m(\xi)=\int_{\mathbb{R}^N}(1-\cos(z\xi))\mathcal{J}(z)\,dz.
\end{equation}
It is  radial  since the kernel $\mathcal{J}$ is radial.
Without imposing any condition on $\mathcal{J}$, apart of being of L\'evy type, we have the trivial estimates $m(0)=0$ and
\begin{equation}\label{minmax}
c_1\min\{1,|\xi|^2\}\le m(\xi)\le c_2\max\{1,|\xi|^2\}.
\end{equation}

In the case of an integrable kernel $\mathcal{J}$ we have $m(\xi)=\|\mathcal{J}\|_1-\widehat{\mathcal{J}}(\xi)$, and then $m( \xi)\le c$. This implies that $\widehat{G}_t\notin L^\sigma(\mathbb{R}^N)$ for any $\sigma<\infty$, $t>0$.

Assume then $\mathcal{J}\notin L^1(B_1)$. If $\mathcal{J}$ is weakly nonintegrable, i.e., $|z|^N\mathcal{J}(z)\le c$ for $|z|\le1$, then using \eqref{J0} we get for $|\xi|>1$
$$
\begin{array}{rl}
m(\xi)&\displaystyle\le \int_{|z|\le|\xi|^{-1}}\frac{c|\xi|^2|z|^2}{|z|^N}\,dz+\int_{|\xi|^{-1}<|z|<1}\frac{c}{|z|^N}\,dz+\int_{|z|\ge1}\mathcal{J}(z)\,dz\\ [3mm]
&\displaystyle\le c_1+c_2\log|\xi|+c_3,
\end{array}
$$
a growth that is again not enough to get $\widehat{G}_t$ in any $L^\sigma(\mathbb{R}^N)$, $\sigma<\infty$ provided $t$ is small. It is then crucial to have a stronger singularity.

Assuming then a power singularity, i.e. condition  \eqref{J3}, we have, by a simple change of variables
\begin{equation}
  \label{rho>1}
  \begin{array}{rl}
m(\xi)&\displaystyle=|\xi|^{-N}\int_{\mathbb{R}^N}(1-\cos z_1)\mathcal{J}(z/|\xi|)\,dz \\ [3mm]
&\displaystyle\ge c|\xi|^{\beta}\int_0^1\frac{1-\cos w}{w^{1+\beta}}\,dw = c|\xi|^{\beta},
\end{array}
\end{equation}
for $|\xi|>1$.
As to the behaviour at the origin of $m$, we have that \eqref{J1} implies
$$
\begin{array}{rl}
m(\xi)&\displaystyle\le \int_{|z|\le1}|\xi|^2|z|^2\mathcal{J}(z)\,dz+\int_{1<|z|<|\xi|^{-1}}\dfrac{|\xi|^2|z|^2}{|z|^{N+\gamma}}\,dz\\ [3mm]
&\displaystyle
+\int_{|z|\ge|\xi|^{-1}}\frac1{|z|^{N+\gamma}}\,dz\le c_1|\xi|^2+c_2\Big||\xi|^2-|\xi|^\gamma\Big|+c_3|\xi|^\gamma.
\end{array}
$$
Thus
\begin{equation}
  \label{rho2<1}
m(\xi)\le c|\xi|^{\overline\gamma}\qquad|\xi|\le1,\quad\overline\gamma=\min\{\gamma,2\}.
\end{equation}
On the contrary, hypothesis  \eqref{J2} implies
$$
m(\xi)=\int_{\mathbb{R}^N}(1-\cos(z\xi))\mathcal{J}(z)\,dz\ge c\int_{c_1|\xi|^{-1}}^{c_2|\xi|^{-1}}|\xi|^2r^{1-\omega}\,dr=c|\xi|^{\omega}.
$$
 for $|\xi|\le1$. Using \eqref{minmax} we deduce
\begin{equation}
  \label{rho<1}
m(\xi)\ge c|\xi|^{\varpi}\qquad |\xi|\le1,\quad \varpi=\min\{\omega,2\}.
\end{equation}
In summary, assuming hypotheses \eqref{J1} (and \eqref{J0}) we get
\begin{equation}
  \label{estimates-m0}
m(\xi)\le\left\{\begin{array}{cl}
c_1|\xi|^{\overline\gamma}&\quad\text{if } |\xi|\le 1, \\ [2mm]
c_2|\xi|^2&\quad\text{if } |\xi|\ge 1.
\end{array}\right.
\end{equation}
and also assuming \eqref{J2} and \eqref{J3}
\begin{equation}
  \label{estimates-m}
m(\xi)\ge\left\{\begin{array}{cl}
c_3|\xi|^\varpi&\quad\text{if } |\xi|\le 1, \\ [2mm]
c_4|\xi|^\beta&\quad\text{if } |\xi|\ge 1.
\end{array}\right.
\end{equation}

\subsection{The heat kernel $G_t$}

Let $\widehat{G}_t(\xi)=e^{-m(\xi)t}$. Clearly $0<\widehat G_t\le1$. Also, an easy comparison principle gives $G_t>0$. In fact, if $w$ is any smooth function satisfying
$$
\left\{
\begin{array}{ll}
(\partial_t +\mathcal{L}_x) w=0,\qquad & \text{in } Q,\\
w(x,0)\ge0,
\end{array}\right.
$$
and if $w(x_0,t_0)=0$ is a minimum of $w$ we would have
$$
\partial_tw(x_0,t_0)\le0
$$
and
$$
\mathcal{L}_x w(x_0,t_0)=\int_{\mathbb{R}^N}(w(x_0,t_0)-w(z,t_0))\mathcal{J}(x_0-z)\,dz<0,
$$
which is a contradiction, and it implies $w>0$ in $Q$. As a corollary $\|G_t\|_1=\widehat{G}_t(0)=1$ for every $t>0$.

Let us estimate now the $L^q$ norm of $G_t$ for $q>1$ using the estimates \eqref{estimates-m} on the symbol $m(\xi)$.

\begin{prop}\label{prop-Gt}
Assume hypotheses \eqref{J2} and \eqref{J3}. For any $1\le q\le\infty$ it holds
\begin{equation}\label{estim-Gt}
  \|G_t\|_q\le \left\{\begin{array}{ll}
  ct^{-\frac N\beta(1-1/q)}&\text{ if } t\le1,\\ [2mm]
  ct^{-\frac N\varpi(1-1/q)}&\text{ if } t\ge1.\end{array}
  \right.
\end{equation}
\end{prop}

\begin{proof}
First
$$
\|G_t\|_\infty\le c\int_{\mathbb {R}^N}e^{-m(\xi)t}\,dx\le c\int_0^1e^{-cr^\varpi t}r^{N-1}\,dt+c\int_1^\infty e^{-cr^\beta t}r^{N-1}\,dt=I_1+I_2.
$$
$$
I_1\le  ct^{-\frac N\varpi}\int_0^te^{-z}z^{\frac N\varpi-1}\,dz\le c\left\{
\begin{array}{ll}
1&\text{ if } t\le1, \\ [2mm]
t^{-\frac N\varpi}&\text{ if } t\ge1,
\end{array}
\right.
$$
$$
I_2\le ct^{-\frac N\beta}\int_t^\infty e^{-z}z^{\frac N\beta-1}\,dz\le c\left\{
\begin{array}{ll}
t^{-\frac N\beta}&\text{ if } t\le1, \\ [2mm]
t^{-1}e^{-t}&\text{ if } t\ge1,
\end{array}
\right.
$$
Thus
$$
\|G_t\|_\infty\le c\left\{
\begin{array}{ll}
t^{-\frac N\beta}&\text{ if } t\le1, \\ [2mm]
t^{-\frac N\varpi}&\text{ if } t\ge1.
\end{array}
\right.
$$
The result follows by interpolation with $\|G_t\|_1=1$.
\end{proof}

This estimate should be complemented with a pointwise estimate, in the spirit of \cite{BlumenthalGetoor60,PruittTaylor69}. This only works  in the case $\beta=\omega\in (0,2)$, see  \cite{BassLevin02}, thus getting
$$
G_t(x)\le c\min\{t^{-N/\beta},t|x|^{-N-\beta}\}.
$$

\section*{Acknowledgments}

Work supported by the Spanish project  PID2020-116949GB-I00. The first author was also supported by Grupo de Investigaci\'on UCM 920894.



\end{document}